\documentclass[a4paper,leqno,11pt]{article}

\usepackage[utf8]{inputenc}
\usepackage[T1]{fontenc}
\usepackage[english]{babel}

\usepackage{geometry}

\usepackage{fancyhdr}

\usepackage{hyperref}
\hypersetup{
  colorlinks   = true, %
  urlcolor     = blue, %
  linkcolor    = blue, %
  citecolor   = black %
}

\usepackage[inline]{enumitem}

\usepackage{graphicx}

\usepackage{color}

\usepackage[intlimits]{amsmath}
\usepackage{amssymb, amsfonts}

\usepackage[thmmarks, amsmath, amsthm]{ntheorem}

\usepackage{MnSymbol}
\usepackage{mathbbol}

\usepackage{bbm}

\usepackage{mathrsfs}

\usepackage[all]{xy}

\numberwithin{equation}{section}
\newtheorem{theoremcounter}{theoremcounter}[section]

\theoremstyle{plain}

\newtheorem{lemma}[theoremcounter]{Lemma}
\newtheorem{proposition}[theoremcounter]{Proposition}

\newtheorem{theorem}[theoremcounter]{Theorem}
\newtheorem{claim}{Claim}

\theoremnumbering{Alph}
\newtheorem{introtheorem}{Theorem}
\newtheorem{introcorollary}[introtheorem]{Corollary}
\theoremstyle{definition}

\theoremstyle{remark}

\newtheorem{example}[theoremcounter]{Example}

\newtheorem{notation}[theoremcounter]{Notation}

\newtheorem{remark}[theoremcounter]{Remark}

\usepackage{xargs}                      %
\usepackage[pdftex,dvipsnames]{xcolor}  %
\usepackage[colorinlistoftodos,prependcaption,textsize=tiny]{todonotes}
\newcommandx{\unsure}[2][1=]{\todo[linecolor=red,backgroundcolor=red!25,bordercolor=red,#1]{#2}}
\newcommandx{\change}[2][1=]{\todo[linecolor=blue,backgroundcolor=blue!25,bordercolor=blue,#1]{#2}}
\newcommandx{\info}[2][1=]{\todo[linecolor=OliveGreen,backgroundcolor=OliveGreen!25,bordercolor=OliveGreen,#1]{#2}}
\newcommandx{\improvement}[2][1=]{\todo[linecolor=Plum,backgroundcolor=Plum!25,bordercolor=Plum,#1]{#2}}

\usepackage[style=alphabetic]{biblatex}
\DeclareFieldFormat{eprint}{\href{https://arxiv.org/abs/#1}{\texttt{#1}}}

\renewbibmacro*{date}{%
\printdate
\iffieldundef{origyear}{%
}{%
  \setunit*{\addspace}%
  \printtext[parens]{\printorigdate}%
}%
}

\newcommand{\cA}{\ensuremath{\mathcal{A}}}

\newcommand{\cN}{\ensuremath{\mathcal{N}}}

\newcommand{\cU}{\ensuremath{\mathcal{U}}}

\newcommand{\cZ}{\ensuremath{\mathcal{Z}}}

\newcommand{\bS}{\ensuremath{\mathbb{S}}}
\newcommand{\bT}{\ensuremath{\mathbb{T}}}

\newcommand{\rL}{\ensuremath{\mathrm{L}}}

\newcommand{\rR}{\ensuremath{\mathrm{R}}}

\newcommand{\rmd}{\ensuremath{\mathrm{d}}}

\newcommand{\veps}{\ensuremath{\varepsilon}}

\newcommand{\vphi}{\ensuremath{\varphi}}

\newcommand{\ol}{\overline}

\newcommand{\eqstop}{\ensuremath{\, \text{.}}}
\newcommand{\eqcomma}{\ensuremath{\, \text{,}}}

\newcommand{\NN}{\ensuremath{\mathbb{N}}}
\newcommand{\ZZ}{\ensuremath{\mathbb{Z}}}

\newcommand{\RR}{\ensuremath{\mathbb{R}}}
\newcommand{\CC}{\ensuremath{\mathbb{C}}}

\newcommand{\id}{\ensuremath{\mathrm{id}}}

\newcommand{\lra}{\ensuremath{\longrightarrow}}

\newcommand{\Aut}{\ensuremath{\mathrm{Aut}}}
\newcommand{\ot}{\ensuremath{\otimes}}
\newcommand{\Cstar}{\ensuremath{\mathrm{C}^*}}

\newcommand{\bo}{\ensuremath{\mathcal{B}}}
\newcommand{\contc}{\ensuremath{\mathrm{C}_\mathrm{c}}}
\newcommand{\Ltwo}{\ensuremath{{\offinterlineskip \mathrm{L} \hskip -0.3ex ^2}}}
\newcommand{\ltwo}{\ensuremath{\ell^2}}

\newcommand{\Rep}{\ensuremath{\mathrm{Rep}}}

\newcommand{\lp}{\ensuremath{\ell^r}}

\newcommand{\sign}{\ensuremath{\mathrm{sign}}}
\newcommand{\qm}{\mathbf{q}}
\newcommand{\am}{\mathbf{a}}
\newcommand{\bm}{\mathbf{b}}
\newcommand{\zm}{\mathbf{z}}
\newcommand{\epsilonm}{\mathbf{\epsilon}}
\newcommand{\vepsm}{{\boldsymbol \veps}}

\newcommand{\tvepsm}{{\tilde {\boldsymbol \veps}}}
\newcommand{\St}{\ensuremath{\mathrm{St}}}
\newcommand{\lang}{\ensuremath{\langle}}
\newcommand{\rang}{\ensuremath{\rangle}}
\newcommand{\pos}{\ensuremath{\mathrm{pos}}}

\newcommand{\authors}{Sven Raum and Adam Skalski}
\renewcommand{\title}{Factorial multiparameter Hecke \mbox{von Neumann algebras} and representations of groups acting on right-angled buildings}
\newcommand{\shorttitle}{Factorial Hecke von Neumann algebras}

\begin{document}

\thispagestyle{empty}

\begin{center}
  \begin{minipage}[c]{\linewidth}
    \textbf{\LARGE \title} \\[0.3em]      
    by \authors
  \end{minipage}
\end{center}

\renewcommand{\thefootnote}{}
\footnotetext{
  \textit{MSC classification:}
  46L65;
  46L10,
  20C08,
  20E42,
  22D10
}
\footnotetext{
  \textit{Keywords:}
  Hecke von Neumann algebra, ${\rm II}_1$ factor, right-angled Coxeter group, right-angled building, strongly transitive action, spherical representation, character space
}

\begin{center}
  \begin{minipage}{\linewidth}
    \vspace{-0.6em}
    \textbf{Abstract}.
    We obtain a complete characterisation of factorial multiparameter Hecke von Neumann algebras associated with right-angled Coxeter groups.  Considering their $\ell^p$-convolution algebra analogues, we exhibit an interesting parameter dependence, contrasting phenomena observed earlier for group Banach algebras.  Translated to Iwahori-Hecke von Neumann algebras, these results allow us to draw conclusions on spherical representation theory of groups acting on right-angled buildings, which are in strong contrast to behaviour of spherical representations in the affine case.  We also investigate certain graph product representations of right-angled Coxeter groups and note that our von Neumann algebraic structure results show that these are finite factor representations.  Further classifying a suitable family of them up to unitary equivalence allows us to reveal high-dimensional Euclidean subspaces of the space of extremal characters of right-angled Coxeter groups.
  \end{minipage}
\end{center}

\section{Introduction}
\label{sec:introduction}

Generic Hecke algebras are deformations of the group ring of a Coxeter group.  Being intimately related to deep representation theoretic problems, they attracted great interest for spherical and affine Coxeter groups \cite{kazhdanlusztig79}.  In recent decades, other Coxeter groups started to attract interest, driven by the theory of buildings and by Kac-Moody groups acting on them \cite{remy12-survey}.  An operator algebraic perspective on Hecke algebras was created by Dymara who introduced Hecke von Neumann algebras $\cN_\qm(W)$ associated with a Coxeter system $(W, S)$ and specific deformation parameters $\qm \in \RR_{>0}^S$ in order to study cohomology of buildings \cite{dymara06, davisdymarajanusykiewiczokun07}.

Among non spherical and non affine Coxeter groups, the right-angled ones are particularly interesting. It was shown by Haglund and Paulin that buildings of arbitrary thickness associated with right-angled Coxeter groups exist \cite{haglundpaulin03}.  Not only certain Kac-Moody groups act on right-angled buildings, but also the buildings' automorphism groups \cite{caprace14} and universal groups with prescribed local action \cite{demedtssilvastruyve18,demedtssilva19} enrich the class of examples.  From an operator algebraic point of view, right-angled Coxeter groups are interesting thanks to their rich combinatorial structure as graph products \cite{caspersfima17}, which makes a variety of operator algebraic tools available.

The first structural result on Hecke von Neumann algebras was obtained by Garncarek, who characterised factorial single parameter Hecke von Neumann algebras of right-angled Coxeter groups \cite{garncarek16}.  Whether the multiparameter generalisation of his result holds remained unclear at the time and was advertised as an open problem in \cite[Question 2]{garncarek16}.  After further advances in the understanding of Hecke operator algebras \cite{capsersskalskiWasilewski19,caspers20},   Caspers, Klisse and Larsen recently obtained in \cite{caspersklisselarsen21} a factoriality result for some multiparameters, as a consequence of their simplicity and unique trace results for Hecke \Cstar-algebras.  However, they state in \cite[Remark 5.8 (a)]{caspersklisselarsen21} that Garncarek's proof ``does not trivially extend to the multiparameter case''.  This statement is indeed true, but here we extend the analytic part of Garncarek's work to the multiparameter case and complement it with a new combinatorial approach in order to obtain a complete solution to \cite[Question 2]{garncarek16}.  At the same time we initiate the study of the corresponding $\lp$-operator algebras.  Our result is most conveniently formulated for parameters  $\qm \in (0,1]^S$, covering the general case thanks to standard considerations, based on Kazhdan-Lusztig's parameter reduction.  Note that throughout the paper we adopt a convention that if $x \in \mathbb{R}$ and $r \in [1, \infty)$ then $x^{\frac{r}{2}}:= \textup{sgn}(x) |x|^{\frac{r}{2}}$.
\begin{introtheorem}
  \label{thm:factoriality}
  Let $(W, S)$ be an irreducible, right-angled Coxeter system with at least three generators, let $\qm \in (0,1]^S$ and let $r \in (1,\infty)$.  Given $\vepsm = (\veps_s)_{s \in S}  \in \{-1,1\}^S$, write $q_{s,\vepsm} = \veps_s q_s^{\veps_s}$ for $s \in S$ and $q_{w, \vepsm} = q_{s_1, \vepsm} \dotsm q_{s_n, \vepsm}$ for a reduced expression $w = s_1\dotsm s_n \in W$. Set $\tilde{r}= \min\{r,r'\}$, where $r'$ is the conjugate exponent of $r$  and consider the set of parameters with convergent weighted growth series $\tilde{C} = \{ \vepsm \in \{-1,1\}^S \mid \sum_{w \in W} q_{w, \vepsm}^{\frac{\tilde{r}}{2}} \text{ converges absolutely} \}$. Then
  \begin{gather*}
    \cN_q^r(W) = M \oplus \bigoplus_{\vepsm \in \tilde{C}} \CC p_\vepsm
  \end{gather*}
  where $M$ is a factor (i.e.\ has a trivial centre) and $p_\vepsm$ is a natural projection onto the one-dimensional space of Hecke eigenvectors with weight $(q^{\frac{r}{2}}_{s,\vepsm})_{s \in S}$.  In particular, $\cN_q^r(W)$ is a factor if and only if $\sum_{w \in W} q_{w,\mathbf{1}}^{\frac{\tilde{r}}{2}} = \infty$.
\end{introtheorem}
Including the $\lp$-convolution algebra analogues $\cN_\qm^r(W)$ of Hecke von Neumann algebras allows us to exhibit the interesting parameter dependence, which is in contrast to simplicity results obtained so far. Indeed, it is easy to see that the classical condition to have only infinite (non-trivial) conjugacy classes characterises factoriality of group convolution algebras on arbitrary $\lp$-spaces. Moreover, the articles \cite{pooyahedjasian15,phillips19} show that also on a \Cstar-algebraic level, simplicity of a classical operator algebra associated with a group implies simplicity of its $\lp$-analogues. Our result shows that for Hecke operator algebras a similar implication holds no longer true.  The remark is due that some simplicity results for the norm-closed $\lp$-analogues of Hecke operator algebras can be obtained from existing literature. Combining the averaging operators introduced in \cite{caspersklisselarsen21} with the interpolating methods employed in \cite{phillips19}, it follows directly that for every Hecke \Cstar-algebra $\Cstar_q(W)$ for which \cite{caspersklisselarsen21} proves simplicity, all norm-closed $\lp$-analogues are simple for $r$ from a neighbourhood of $2$.  The crucial difference to \cite{phillips19}, preventing further conclusions, is the lack of a good norm-estimate for the averaging operators on $\ell^1(W)$.

It must also be mentioned that the infinite dimensional von Neumann factors arising in Theorem~\ref{thm:factoriality} are non-injective by \cite[Theorem 6.2]{caspersklisselarsen21}.  Furthermore, an observation made already in \cite[Section 6]{garncarek16} (compare also with \cite[Section 5.4]{caspersklisselarsen21}) is that Dykema's free probability results \cite{dykema93-hyperfinite} apply in the special case of free products $W = \ZZ/2\ZZ^{*n}$ and yield there  the von Neumann algebraic statement of our Theorem~\ref{thm:factoriality}.  In the present treatment Dykema's results can be used to give an alternative proof in the von Neumann algebraic case for the base case of an induction.  Compare with Remark \ref{rem:dykema}.

It is fair to say that recent literature treated Hecke operator algebras as analytic-combinatorial objects.  Concerning our main methods, the present work is no different.  This is not the case when applications are concerned.  The origin of Hecke operator algebras is geometric and lies in the Singer conjecture on cohomology of buildings.  But already Garncarek has to state that ``it turns out, although the centers of $\cN_q(W)$ can be nontrivial, they contribute nothing new in the subject of decomposing the weighted cohomology of $W$'' \cite[p.1203]{garncarek16}.  Subsequent work on Hecke operator algebras makes no different claim, and also our work does not create new insight into weighted cohomology.  However, connections between Hecke operator algebras and the representation theory of groups acting on buildings have so far not appeared in the literature.  A classical result in representation theory of p-adic reductive groups going back to Iwahori and Matsumoto \cite{iwahorimatsumoto65} says that generic Hecke algebras are isomorphic with double coset Hecke algebras associated with Iwahori subgroups.  By definition, an Iwahori subgroup of a p-adic reductive group is the stabiliser of a chamber in the associated Bruhat-Tits building.  This identification of Hecke algebras continues to hold for groups acting strongly transitively on arbitrary buildings.  We are hence able to deduce information about the spherical representation theory of certain groups acting on right-angled buildings from (the Hilbert space case of) Theorem~\ref{thm:factoriality}.  The most precise way to do so, is to describe a decomposition of the quasi-regular representation associated with an Iwahori subgroup.  It is unitarily equivalent to a restriction of the regular representation, since Iwahori subgroups are compact.  The next corollary applies to several classes of groups mentioned before, such as certain completed Kac-Moody groups \cite{remyronan06,remy12-survey}, the full group of type preserving automorphism of right-angled buildings \cite{caprace14} and certain universal groups with prescribed local action \cite{demedtssilvastruyve18,demedtssilva19}.
\begin{introcorollary}
  \label{cor:spherical-representations}
  Let $G \leq \Aut(X)$ be a closed subgroup acting strongly transitively and type preserving on a right-angled irreducible thick building $X$ of type $(W, S)$ and rank at least three.  Denote by $(d_s)_{s \in S}$ the thickness of $X$, let $B \leq G$ be the stabiliser of a chamber in $X$ and let $K \leq G$ be the stabiliser of a vertex in $X$.  Given $\vepsm = (\veps_s)_{s \in S}  \in \{-1,1\}^S$, write $q_{s,\vepsm} =  \veps_s d_s^{-\veps_s}$ for $s \in S$ and $q_{w, \vepsm} = q_{s_1, \vepsm} \dotsm q_{s_n, \vepsm}$ for a reduced expression $w = s_1\dotsm s_n \in W$.  Denote $C = \{ \vepsm \in \{-1,1\}^S \mid \sum_{w \in W} q_{w, \vepsm}^{\frac{r}{2}} \text{ converges absolutely} \}$.
  \begin{itemize}
  \item The quasi-regular representation $\lambda_{G,B}$ on $\ltwo(G/B)$ decomposes as a direct sum of a type ${\rm II}_\infty$ factor representation and pairwise unitarily inequivalent, infinite dimensional, irreducible unitary representations $\St_{\vepsm}$, $\vepsm \in C$.  In particular, if $\sum_{w \in W} q_{w, \textbf{1}} = \infty$ then $\lambda_{G,B}$ is type ${\rm II}_\infty$ factor representation.
  \item If $(d_s)_{s \in S}$ is constant, then $\lambda_{G,K}$ is a type ${\rm II}_\infty$ factor representation.
  \end{itemize}
\end{introcorollary}
This result provides a complete description of square integrable Iwahori-spherical representations.  The existence of the irreducible direct summand $\St_{\mathbf{1}}$ of $\lambda_{G,B}$ under the condition that $\sum_{w \in W} q_{w,\textbf{1}} < \infty$ follows from the existence of the Steinberg representations, already described by Matsumoto in \cite[Number 1.3]{matsumoto69}.  The representations $\St_{\vepsm}$ appearing in Corollary \ref{cor:spherical-representations} are analogues of the classical Steinberg representation in as far as they describe all Iwahori-spherical discrete series representations and they are associated with Hecke eigenvectors, as is apparent from the proof of the corollary.  Factoriality of the orthogonal complement of the Steinberg representations is in stark contrast to previously observed behaviour of spherical representation theory.  In the classical cases, such as groups acting strongly transitively and type preserving on trees \cite{cartier73,figa-talamancanebbia91} and more generally on affine buildings \cite{matsumoto77,capraceciobotaru15}, the Iwahori-Hecke algebra associated with $B$ is a finitely generated module over the abelian spherical Hecke algebra, which is associated with the maximal compact subgroup $K$.  It was already shown in \cite[Theorem 3.5]{lecureux10} (see also \cite{abramenkoparkinsonvanmaldeghem13} and \cite[Corollary 2]{monod19-gelfand}) that the spherical Hecke algebra in the setting of our Corollary \ref{cor:spherical-representations} cannot be abelian.  The decomposition of the quasi-regular representation we provide, however, goes far beyond the statement of the spherical Hecke algebra being non-abelian, and similar examples were not obtained before.  In particular, it follows already from \cite{suzuki17} (see also \cite[Theorem~G]{raum15-powers} whose proof is not correct as stated, but the result can be recovered based on Suzuki's article) that there are groups acting on trees whose regular representation is factorial, but the group action is not (and cannot be) strongly transitive.

In the second part of this article, we focus on unitary representation theory of right-angled Coxeter groups.  In \cite[Notes 19.2]{davis08} and \cite[Proposition 4.2]{caspersklisselarsen21} it was observed that generic Hecke algebras associated with right-angled Coxeter groups are isomorphic with the respective group algebras.  A similar phenomenon is known in the spherical case \cite[Chapter IV, Exercise 27]{bourbaki-lie-groups-lie-algebras}.  We put this observation into a systematic context of graph products of representations.  Thanks to Theorem \ref{thm:factoriality} we can single out a family $\lambda_\am$, $\am \in (-1,1)^S$ of finite factor representations arising from graph product representations of an irreducible, right-angled Coxeter system on at least three generators.  Our second main result shows that close to the regular representation these representations are pairwise not unitarily equivalent, even though they sometimes generate the same Hecke von Neumann algebra. We expect that in fact they are  pairwise non-equivalent for all $\am \in (-1,1)^S$ and provide further results supporting this belief.  This is an interesting phenomenon in view of the open classification problems for Hecke operator algebras and it has non-trivial consequences for the character space of right-angled Coxeter groups.
\begin{introtheorem}
  \label{thm:unitary-equivalence}
  Let $(W,S)$ be an irreducible, right-angled Coxeter system with at least three generators.  There are finite factor representations $\lambda_\am$, $\am = (a_s)_{s \in S}\in (-1, 1)^S$ of $W$ such that $\cN_\qm(W) = \lambda_\am(W)'' \oplus A_\am$ for $q_s = \frac{1 - |a_s|}{\sqrt{1 - a_s^2}}$ and some finite dimensional abelian von Neumann algebra $A_\am$ depending on $\am$.  Further, there is an open set $U \subset (-1,1)^S$ containing $\bf{0}$ such that the representations $(\lambda_\am)_{\am \in U}$ are pairwise unitarily inequivalent.
\end{introtheorem}

Combining the continuous parameter dependence obtained from the construction as graph product representations, Theorem \ref{thm:unitary-equivalence} implies the existence of high-dimensional Euclidean subspaces of the character space of right-angled Coxeter groups.   This result is in contrast to recent classification results for traces on group \Cstar-algebras obtained for example in \cite{bekka07-superrigidity,petersonthom14,creutzpeterson2013-character-rigidity,boutonnethoudayer19,bekka19,bekkafrancini20,baderboutonnethoudayerpeterson20}.  While graph products of groups are expected to have complicated trace spaces, their precise topological structure remains unclear.
\begin{introcorollary}
  \label{cor:traces}
  Let $(W,S)$ be an irreducible, right-angled Coxeter system with at least three generators.  There is an $|S|$-dimensional Euclidean subspace of the extremal points of $\mathrm{Char}(W)$ whose interior contains the regular trace.
\end{introcorollary}

This article contains four sections.  After the introduction, we collect some preliminaries on the various topics we treat.  In Section \ref{sec:factoriality}, we prove our first main Theorem \ref{thm:factoriality} and its Corollary \ref{cor:spherical-representations}.  In Section \ref{sec:graph-products-representations}, we introduce the graph products of representations and prove Theorem \ref{thm:unitary-equivalence} and Corollary \ref{cor:traces}.

\subsection*{Acknowledgements}

S.R.\ was supported by the Swedish Research Council through grant number 2018-04243 and by the European Research Council (ERC) under the European Union's Horizon 2020 research and innovation programme (grant agreement no. 677120-INDEX).
A.S.\ was partially supported by the National Science Center (NCN) grant no.~2020/39/I/ST1/01566. 

The authors are deeply grateful to Pierre-Emmanuel Caprace for pointing out that the factor representation in Corollary \ref{cor:spherical-representations} cannot be finite, and for clarifying remarks on the representation theory of reductive groups.  His remarks also led us to state the factoriality result for spherical representations.  We thank Mateusz Wasilewski for useful comments improving the exposition of our work. We also thank Pierre-Emmanuel Caprace and Bertrand R{\'e}my for pointing out the references \cite{abramenkoparkinsonvanmaldeghem13} and \cite{lecureux10}. Finally we are very grateful to the anonymous referee for a very careful reading of our paper and suggesting several small improvements of presentation.

\section{Preliminaries}
\label{sec:preliminaries}

\subsection{Locally compact groups, convolution algebras and regular representations}
\label{sec:representations-locally-compact}

Standard resources for this material are \cite{deitmarechterhoff14} and the appendices of \cite{bekkadelaharpevalette08}.  Every locally compact group $G$ admits a left invariant Radon measure $\mu$, which is unique up to a positive scalar.  The modular function of $G$ is the unique group homomorphism $\Delta: G \to \RR_{>0}$ satisfying $\Delta(x) \mu(B) = \mu(Bx)$ for all measurable $B$ and all $x \in G$.

The space of continuous and compactly supported functions $\contc(G)$ becomes a *-algebra when equipped with the product
\begin{gather*}
  f*g(x) = \int_G f(y)g(y^{-1}x) \rmd \mu(y)
\end{gather*}
and the involution
\begin{gather*}
  f^*(x)  = \Delta(x^{-1}) \ol{f(x^{-1})}
  \eqstop
\end{gather*}
The convolution product of $\contc(G)$ extends to a *-representation $\contc(G) \to \bo(\Ltwo(G))$.  Given a compact open subgroup $K \leq G$, its \emph{(coset) Hecke algebra} is the *-subalgebra $\CC[G,K] \subset \contc(G)$ of $K$-biinvariant functions.  The Hecke algebra can be described as $\CC[G,K] = p_K \contc(G) p_K$ where $p_K = \mathbb{1}_K$ is the indicator function of $K$; we shall always normalise the Haar mesaure of $G$ on a suitable compact open subgroup $K$.

The \emph{left regular representation} is the the unique group homomorphism $\lambda: G \to \cU(\Ltwo(G))$ satisfying $\lambda_x(f)(y) = f(x^{-1}y)$ for all $f \in \contc(G) \subset \Ltwo(G)$ and all $x,y \in G$.  The \emph{group von Neumann algebra} of $G$ is the double commutant
\begin{gather*}
  \rL(G) = \lambda(G)'' = \ol{\contc(G)}^{\mathrm{SOT}}
  \eqcomma
\end{gather*}
where SOT denotes the strong operator topology, that is the topology of pointwise convergence.  Given a closed subgroup $H \leq G$, we denote by $\lambda_{G, H}: G \to \Ltwo(G/H)$ the \emph{quasi-regular representation}.  If $H \leq G$ is open, $G/H$ is discrete and $\lambda_{G, H}$ is the permutation representation of the left translation action on $G/H$.  In particular, it has no finite dimensional subrepresentations if $G/H$ is infinite.  For a compact open subgroup $K \leq G$, the quasi-regular representation $\lambda_{G,K}$ is unitarily equivalent to the restriction of $\lambda$ to the subrepresentation generated by $K$-fixed vectors.  The \emph{Hecke von Neumann algebra} of the pair $K \leq G$ is
\begin{gather*}
  \rL(G,K) = p_K \rL(G) p_K = \ol{\CC[G,K]}^{\mathrm{SOT}} \subset \bo(\ltwo(K \backslash G))
  \eqstop
\end{gather*}
Both the group von Neumann algebra and the Hecke von Neumann algebra admit counterparts $\rR(G)$ and $\rR(G, K)$ constructed from the \emph{right regular representation} of $G$, which is defined by $\rho_x(f)(y) = \Delta_G(x)^{\frac{1}{2}} f(yx)$ for all $f \in \contc(G) \subset \Ltwo(G)$ and all $x,y \in G$.  Since $\lambda$ and $\rho$ are unitarily equivalent, we have $\rL(G) \cong \rR(G)$ and $\rL(G,K) \cong \rR(G,K)$.

A unitary representation $\pi: G \to \cU(H)$ is a subrepresentation of $\lambda^{\oplus \kappa}$ for some cardinal $\kappa$ if and only if it is \emph{square integrable} in the sense that for a dense set of $\xi \in H$ the function $g \mapsto \langle \pi(g) \xi, \xi \rangle$ is square integrable with respect to the Haar measure of $G$.  See \cite[Chapter 14]{dixmier69-representations} or \cite[Theorem 4.6]{rieffel69}.  In order to avoid a clash of terminology, we prefer the term ``square integrable'' to ``discrete series'' here, since the latter usually makes the additional assumption of irreducibility in the representation theory of reductive groups.

\subsection{Coxeter systems}
\label{sec:coxeter-systems}

The books \cite{humphreys90} and \cite{davis08} both describe Coxeter groups.  A \emph{Coxeter matrix} over a non-empty set $S$ is a symmetric $S \times S$ matrix whose diagonal entries equal one and whose off diagonal entries lie in $\NN_{\geq 2} \cup \{\infty\}$.  Let $M = (m_{st})_{s,t \in S}$ be a Coxeter matrix.  The \emph{Coxeter system} $(W, S)$ associated with $M$ is given by the group 
\begin{gather*}
  W  = \langle S \mid \forall s,t \in S: \, e = (st)^{m_{st}} \rangle
\end{gather*}
together with its generating set $S$.
The Coxeter matrix and its Coxeter system are called \emph{right-angled} if $m_{st} \in \{2, \infty\}$ for all different $s,t \in S$. They are \emph{irreducible} if the Coxeter matrix is not the Kronecker tensor product of two non-trivial Coxeter matrices.  Equivalently, there is no partition $S = S_1 \sqcup S_2$ such that $W = \langle S_1 \rangle \oplus \langle S_2 \rangle$.

Given $w \in W$ we denote by $|w|$ its \emph{word length} with respect to the generating set $S$.  A word $w$ with letters in $S$ is called \emph{reduced} if it has $|w|$ many letters.  For a right-angled $W$ and $s \in S$ the \emph{$s$-length} $\ell_s(w)$ of $w \in W$ is the number of occurrences of $s$ in any reduced word representing $w$.  Tits' solution to the word-problem in Coxeter groups \cite[Th{\'e}or{\`e}me 3]{tits69} provides normal forms for elements in $W$ and shows that the $s$-length is well defined. Similarly, if for a given word $w\in W$ and $s \in S$ we have $\ell_s(w)\geq 1$ we can define the \emph{position of $s$ in $w$}, $\pos_s(w)$, as the minimal position at which $s$ can occur in a reduced expression for $w$; we shall also say that $w$ \emph{starts with $s$} if $\pos_s(w)=1$ and that \emph{$w$ ends with $s$} if $\pos_s(w^{-1})=1$. Given a subgroup $D = \langle T \rangle$ for $T \subset S$, we call a double coset $DwD \subset W$ \emph{non-degenerate} if it is not equal to a simple coset, i.e.\ to $wD$.

Given a subset $T \subset S$, we denote by $W|_T$ the Coxeter group generated by the elements of $T$.  Given $s \in S$, for simplicity, we write $S \setminus s$ to denote $S \setminus \{s\}$.

\subsection{Generic Hecke algebras}
\label{sec:generic-hecke-algebras}

The book \cite{humphreys90} introduces generic Hecke algebras.  Also \cite{davis08} introduces Hecke algebras and discusses their von Neumann algebraic completions.  Given a Coxeter system $(W, S)$, we denote its real \emph{deformation parameters} by
\begin{gather*}
  \RR_{> 0}^{(W,S)} = \{ \qm:=(q_s)_{s \in S} \in \RR_{> 0}^S \mid \forall s,t \in S: (\exists w \in W: wsw^{-1} = t \implies q_s = q_t)\}
  \eqstop
\end{gather*}
Frequently, the expression
\begin{gather*}
  p_s = \frac{q_s - 1}{\sqrt{q_s}}
\end{gather*}
will be used.  If $(W, S)$ is right-angled, Tits' solution to the word-problem implies that $\RR_{>0}^{(W, S)} = \RR_{>0}^S$.  Indeed, if $s,t, s_1, \dotsc, s_n \in S$ satisfy $s_1 \dotsm s_n s s_n \dotsm s_1 = t$, then comparing the length of both sides, it follows that $[s, s_i] = e$ for all $i \in \{1, \dotsc, n\}$ and thus $s = t$.  For $\qm \in \RR_{>0}^{(W, S)}$ and a reduced word $w = s_1 \dotsm s_n \in W$ we write $q_w = q_{s_1} \dotsm q_{s_n}$.  If $\vepsm \in \{-1,1\}^S$, we also write $q_{s,\vepsm} = \varepsilon_s q_s^{\varepsilon_s}$ and $q_{w, \vepsm} =  q_{s_1, \vepsm} \dotsm q_{s_n, \vepsm}$. We will also use the slightly non-standard convention $q_{w, \vepsm}^{\frac{1}{2}}:= \textup{sgn}(q_{w, \vepsm}) |q_{w, \vepsm}|^{ \frac{1}{2}}$. Note that $q_w = q_{w,\textbf{1}}$.

The \emph{(generic) Hecke algebra} associated with $(W, S)$ and the multiparameter $\qm \in \RR_{> 0}^{(W,S)}$ is the *-algebra with the presentation
\begin{align*}
  \CC_\qm[W]  = \bigl \langle (T_s^{(\qm)})_{s \in S} \mid & \forall s,t \in S, m_{st} < \infty: \, T_s^{(\qm)}T_t^{(\qm)} \dotsm = T_t^{(\qm)}T_s^{(\qm)} \dotsm \, \text{($m_{st}$ many factors)} \eqcomma \\
                                                 & (T_s^{(\qm)} - q_s^{\frac{1}{2}})(T_s^{(\qm)} + q_s^{-\frac{1}{2}})=0 \eqcomma \\
                                                 & (T_s^{(\qm)})^* = T_s^{(\qm)} \bigr \rangle
  \eqstop
\end{align*}
For a reduced expression $w = s_1 \dotsm s_n \in W$, we write $T_w^{(\qm)} = T_{s_1}^{(\qm)} \dotsm T_{s_n}^{(\qm)}$.   One can check that this does not depend on the choice of the reduced expression.  In particular, $(T_w^{(\qm)})^* = T_{w^{-1}}^{(\qm)}$ holds for all $w \in W$.  Further, for all $w,w' \in W$ satisfying $|w w'| = |w| + |w'|$ we have $T_{ww'}^{(\qm)} = T_w^{(\qm)} T_{w'}^{(\qm)}$.

For every $r \in [1,\infty]$, the algebra $\CC_\qm[W]$ admits a representation by bounded operators on $\lp(W)$ satisfying the following formula on its generators.
\begin{gather*}
  \pi_\qm^r(T_s^{(\qm)}) \delta_w
  =
  \begin{cases}
    \delta_{sw} & \text{if } |sw| > |w|, \\
    \delta_{sw} + p_s \delta_{w} & \text{if } |sw| < |w| \eqstop    
  \end{cases}
\end{gather*}
For $r \in [1, \infty]$ we define the Hecke $\lp$-convolution algebra by
\begin{gather*}
  \cN_\qm^r(W) = \pi_\qm^r(\CC_\qm[W])''
  \eqstop
\end{gather*}
Using the isometric isomorphism defined by $J: \lp(W) \lra \lp(W): \delta_w \mapsto \delta_{w^{-1}}$, one defines another representation of $\CC_q[W]$ on $\lp(W)$ by $\rho_\qm^r(T_w^{(\qm)}) = J \pi_\qm^r(T_w^{(\qm)}) J$.  It satisfies
\begin{gather*}
  \rho_\qm^r(T_s^{(\qm)}) \delta_w
  =
  \begin{cases}
    \delta_{ws} & \text{if } |ws| > |w|, \\
    \delta_{ws} + p_s \delta_{w} & \text{if } |ws| < |w| \eqstop    
  \end{cases}
\end{gather*}
A short calculation using $\rho_\qm$ shows then that the vector $\delta_e \in \lp(W)$ is cyclic and separating for $\pi_\qm(\CC_\qm[W])$.  This allows us to study elements $x \in \cN_\qm^r(W)$ by means of their images $\hat x = x \delta_e \in \lp(W)$ and  shows injectivity of $\pi_\qm^r$.  We will henceforth often suppress the representations $\pi_\qm^r$ and $\rho_\qm^r$ in our notation, identifying $T_w^{(\qm)}$ with its image in $\bo(\lp(W))$ and making use of the $\CC_\qm[W]$-bimodule structure of $\lp(W)$.  We call a vector $\xi \in \lp(W)$ \emph{Hecke central (for the parameter $\qm$)} if $T_w^{(\qm)}\xi = \xi T_w^{(\qm)}$ for all $w \in W$.

The next lemma in the case of $r=2$ can be deduced in a standard way from \cite[Proposition 19.2.1]{davis08}; for non-deformed convolution algebras associated with locally compact groups the corresponding result can be found in \cite{dawsspronk19}.

\begin{lemma}\label{commutants}
Let $r \in (1, \infty)$. Then $\pi_\qm^r(\CC_\qm[W])' = \rho_\qm^r(\CC_\qm[W])''$.	Thus $T \in B(\ell^r(W))$ belongs to the center of $\cN_\qm^r(W)$ if and only if $T \in \pi_\qm^r(\CC_\qm[W])' \cap\rho_\qm^r(\CC_\qm[W])'$. If this is the case, $T\delta_e$ is a central vector in 	$\ell^r(W)$.
\end{lemma}

\begin{proof}
We only sketch the argument as it follows standard lines, using the duality between $\ell^r(W)$ and $\ell^{r'}(W)$  and the isometry $J:\ell^r(W)\to \ell^{r}(W)$, $J \delta_w = \delta_{w^{-1}}, w \in W$. It suffices to show that if $T \in \pi_\qm^r(\CC_\qm[W])'$ and 
$S \in \rho_\qm^r(\CC_\qm[W])' = J \pi_\qm^r(\CC_\qm[W])'J$, then $TS = ST$. Let $\eta:=T\delta_e \in \ell^r(W)$. 
Then a direct verification shows that for each $w \in W$ we have $(T^* \delta_e)_w = \eta_{w^{-1}}$; thus in fact $\eta \in \ell^{r}(W) \cap \ell^{r'}(W)$ and moreover one can approximate $\eta$ simultaneously in both $\|\cdot\|_r$ and $\|\cdot\|_{r'}$ norms (i.e.\ in $\|\cdot\|_{\min{r,r'}}$) via finitely supported vectors. This implies that for a fixed $w, v \in W$ and $\epsilon>0$ we can find an element $T_0 \in \rho_\qm^r(\CC_\qm[W])$ such that $\|(T-T_0) \delta_w \|_r \|\lambda_{\qm}^r(T_w)\|_{B(\ell^r(W))} \leq \epsilon$, $\|(T^*-T_0^*) \delta_v \|_{r'} \|\rho_{\qm}^{r'}(T_v)\|_{B(\ell^{r'}(W))} \leq \epsilon$. Repeating the same argument for $S$ we can approximate it in an analogous way by an element of $\lambda_\qm^r(\CC_\qm[W])$ and repeating the computation from \cite[Theorem 1.2]{dawsspronk19} yields the equality $((ST)(\delta_w))_v = ((TS)(\delta_v))_w$, so also the first statement of the lemma. The other claims follow easily.
\end{proof}

In the Hilbert space case $r = 2$, the algebra $\cN_\qm(W) = \cN_\qm^2(W)$ is the \emph{multiparameter Hecke von Neumann algebra} associated with $(W,S)$ and $\qm \in \RR_{> 0}^{(W,S)}$.  The representation $\pi_\qm^2$ is a *-representation and it arises from the GNS-construction with respect to the tracial state  $T^{(\qm)}_w \mapsto \langle T^{(\qm)}_w \delta_e , \delta_e \rangle = \delta_{w,e}$ on $\CC_\qm[W]$.

The following parameter reduction is well known from Kazhdan-Lusztig's work \cite{kazhdanlusztig79}.  See also \cite[Proposition 4.7]{caspersklisselarsen21} for the proof of an operator algebraic version for $r=2$. The statement continues to hold in the $\lp$-setting, as can be shown via a direct computation.
\begin{proposition}
  \label{prop:parameter-reduction-hecke-algebra}
  Let $(W,S)$ be a right-angled Coxeter system, let $\qm \in \RR_{>0}^S$, $(\epsilon_s)_{s \in S} \in \{-1,1\}^S$ and $r \in [1,\infty)$.  Define $\qm' = (q_s^{\epsilon_s})_{s \in S}$.  Then the map
  \begin{gather*}
    \CC_\qm[W] \lra \CC_{\qm'}[W]: T_s^{(\qm)} \mapsto \epsilon_s T_s^{(\qm')}
  \end{gather*}
  is well defined and extends to an isometric isomorphism $\cN_\qm^r(W) \cong \cN_{\qm'}^r(W)$.
\end{proposition}

A key notion in this paper is that of \emph{(left) Hecke eigenvectors (for the parameter $\qm$)}, that is non-zero vectors $\eta \in \lp(W)$ such that $T_s^{(\qm)}\eta \in \CC\eta$ for all $s \in S$.  The scalars $(c_s)_{s \in S}$ such that $T_s^{(\qm)} \eta = c_s \eta$ for all $s \in S$ are called the weight of $\eta$.  Already in \cite[Section 5]{davisdymarajanusykiewiczokun07}, some natural idempotents in Hecke von Neumann algebras were recognised, one of which is central and corresponds to a Hecke eigenvector.  See also \cite[Proof of Theorem 5.3]{garncarek16}.  %
The next result provides a complete description of Hecke eigenvectors for right-angled Coxeter systems. 
\begin{proposition}
  \label{prop:hecke-eigenvectors}
  Let $(W,S)$ be a right-angled Coxeter system with at least three generators and let $\qm \in (0,1]^S$.  The formal series $\eta_\vepsm = \sum_{w \in W} q_{w, \vepsm}^{\frac{1}{2}} \delta_w$ is a left and right Hecke eigenvector for the parameter $\qm$ whose weights are given by $T_s^{(\qm)} \eta_\vepsm = \eta_\vepsm  T_s^{(\qm)} = \veps_s q_s^{\veps_s \frac{1}{2}} \eta_\vepsm$ for $s \in S$. Thus in particular $\eta_{\vepsm}$ is central.

  Further, given $r \in [1, \infty)$ and writing $C = \{\vepsm \in \{-1, 1\}^S \mid \sum_w q_{w, \vepsm}^{\frac{r}{2}} \text{ converges absolutely}\}$, Hecke eigenvectors in $\ell^r(W)$ for the parameter $\qm$ are precisely non-zero scalar multiples of the vectors $\eta_\vepsm$ for $\vepsm \in C$.  In particular, there is a Hecke eigenvector in $\lp(W)$ if and only if $\sum_{w \in W} q_w ^{\frac{r}{2}} < \infty$. %
\end{proposition}
\begin{proof}
  A short calculation shows that $\eta_\vepsm$ is indeed an eigenvector with weight $(\veps_s q_s^{\veps_s \frac{1}{2}})_{s \in S}$. 
  Indeed, to show that for each $w \in W$, $s \in S$
 \[(T_s \eta_{\vepsm})_w = \veps_s q_s^{\veps_s \frac{1}{2}} (\eta_{\vepsm})_w \]
we need to  consider two cases: if $|sw|>|w|$ then we require that $q_{sw, \vepsm}^{\frac{1}{2}} =  \veps_s q_s^{\veps_s \frac{1}{2}} q_{w, \vepsm}^{\frac{1}{2}}$, which is a direct consequence of the definition of the coefficients $q_{w, \vepsm}$. On the other hand if $|sw|<|w|$, we require that 
$q_{sw, \vepsm}^{\frac{1}{2}} + p_s q_{w, \vepsm}^{\frac{1}{2}} =  \veps_s q_s^{\veps_s \frac{1}{2}} q_{w, \vepsm}^{\frac{1}{2}}$, so equivalently $\veps_s q_s^{-\veps_s \frac{1}{2}} + p_s =  \veps_s q_s^{\veps_s \frac{1}{2}}$. The last equality can be checked separately for $\veps_s = \pm 1$, so we deduce that the displayed formula holds. An analogous computation shows that $\eta_{\vepsm}$ is also a right Hecke eigenvector with the same weight.
  
   We also note that the formula $\sum_{w \in W} q_{w, \vepsm}^{\frac{1}{2}} \delta_w$ defines a vector in $\lp(W)$ if and only if $\vepsm \in C$.  So we have to show that every eigenvector is of this form.

The relation $(T_s^{(\qm)} - q_s^{\frac{1}{2}}) (T_s^{(\qm)} + q_s^{-\frac{1}{2}}) = 0$ shows that the spectrum of $T_s^{(\qm)}$ consists exactly of $q_s^{\frac{1}{2}}$ and $-q_s^{-\frac{1}{2}}$. So for every Hecke eigenvector $\eta \in \lp(W)$ there is some $\vepsm \in \{-1,1\}^S$ such that $T_s^{(\qm)} \eta = \veps_s q^{\veps_s \frac{1}{2}} \eta$ for all $s \in S$.

  We fix $\vepsm \in \{-1,1\}^S$ and calculate with a formal series $\eta = \sum_{w \in W} x_w \delta_w$ with coefficients in $\CC$ to find that
  \begin{align*}
    T_s^{(\qm)} \sum_{w \in W} x_w \delta_w
    & =
      \sum_{|sw| > |w|} x_w \delta_{sw} + \sum_{|sw| < |w|} x_w \delta_{sw} + p_s x_w \delta_w \\
    & =
      \sum_{|sw| > |w|} x_{sw} \delta_w + \sum_{|sw| < |w|} (x_{sw} + p_sx_w) \delta_w
      \eqstop
  \end{align*}
  Hence $T_s^{(\qm)} \eta = \veps_s q_s^{\veps_s \frac{1}{2}} \eta$ implies that $x_{sw} + p_s x_w = \veps_s q_s^{\veps_s \frac{1}{2}} x_w$ for all $w \in W$ satisfying $|sw| < |w|$. This in turn implies that for such $w$
  \begin{align*}
    x_{sw}
    & =
      x_w(\veps_s q_s^{\veps_s \frac{1}{2}} - p_s) \\
    & =
      x_w\bigl (\veps_s q_s^{\veps_s \frac{1}{2}} - \frac{q_s - 1}{\sqrt{q_s}} \bigr ) \\
    & =
      \begin{cases}
        x_w \bigl ( \frac{q_s - q_s + 1}{\sqrt{q_s}} \bigr ) & \text{if } \veps_s = 1 \\
        x_w \bigl ( \frac{-1 - q_s + 1}{\sqrt{q_s}} \bigr ) & \text{if } \veps_s = -1 \\
      \end{cases} \\
    & =
      x_w \veps_s q_s^{-\veps_s \frac{1}{2}}
      \eqstop
  \end{align*}
  Since $w$ starts with $s$ if and only if $|sw| < |w|$, taking inverses implies that $x_{sw} = \veps_s q_s^{\veps_s \frac{1}{2}} x_w$ for all $w \in W$ satisfying $|sw| > |w|$.  So indeed every Hecke eigenvector with weight $(\veps_s q_s^{\veps_s \frac{1}{2}})_{s \in S}$ is a non-zero multiple of $\eta_\veps$.

\end{proof}

We shall frequently use the notation $\eta_\vepsm$ for a formal series $\sum_{w \in W} q_{w, \vepsm}^{\frac{1}{2}} \delta_w$ for all $\vepsm \in \{-1,1\}^S$. Finally note that as the computation in the last part of the proof above remains valid for $r=\infty$ we deduce that the full list of Hecke eigenvectors in $\ell^\infty(W)$ is  given by  $\{\eta_\vepsm:\vepsm \in \{-1, 1\}^S:\veps_s=1 \textup{ if } q_s<1\}$.

\subsection{Buildings and Iwahori-Hecke algebras}
\label{sec:buildings}

The standard reference for buildings is \cite{abramenkobrown08}.  We follow Garrett's presentation from \cite{garrett97} as it leads directly to the comparison of generic Hecke algebras and Iwahori-Hecke algebras.  A \emph{chamber complex} is a simplicial complex $X$ such that all simplices of $X$ are contained in a maximal simplex and such that for every pair $x,y$ of maximal simplices of $X$ there is a chain $x = x_0, x_1, \dotsc, x_n = y$ of adjacent maximal simplices in $X$.  Its maximal simplices are called \emph{chambers}.  They will usually be denoted by $C$.  A chamber complex is \emph{thin} if each facet of a chamber is the facet of exactly two chambers, and it is called \emph{thick} if each facet of a chamber is the facet of at least three chambers.  An important example of a thin chamber complex is the \emph{Coxeter complex} $\Sigma(W,S)$ associated with a Coxeter system $(W, S)$.  Its simplices are indexed by the cosets $W/ \langle T \rangle$, where $T \subset S$ is a proper subset.  The face relation is determined by opposite inclusion.  In particular, the singletons $w \langle \emptyset \rangle = \{w\}$ represent the maximal simplices of $\Sigma(W,S)$.

A \emph{system of apartments} in a chamber complex $X$ is a set $\cA$ of thin chamber subcomplexes such that
\begin{itemize}
\item for each pair of simplices $x,y$ of $X$ there is $A \in \cA$ such that $x,y \in A$, and
\item if $A, A' \in \cA$ contain a common chamber $C$ and a simplex $x$, then there is an isomorphism of chamber complexes $\phi: A \to A'$ that fixes $C$ and $x$ point-wise.
\end{itemize}
Note that in the last axiom \cite[pp.173--174]{abramenkobrown08} is phrased in terms of a pair of simplices $A$ and $B$, and not a chamber $C$ and a simplex $x$, and that the authors there assume a priori that apartments are Coxeter complexes; both the definitions are in fact equivalent.
A \emph{thick building} is a thick chamber complex that admits a system of apartments.  One can prove that for a thick building $X$ there is a unique maximal apartment system.  So if $X$ is thick, as we assume from now on, it makes sense to speak of an \emph{apartment} of $X$ without reference to a specific system of apartments.  There is a Coxeter system $(W, S)$ such that every apartment of $X$ is chamber isomorphic with $\Sigma(W,S)$.  The choice of a \emph{fundamental chamber} $C$ in $X$ and a compatible labelling of its facets by elements of $S$ makes these isomorphisms unique and introduces a \emph{labelling} of all facets of chambers in $X$.  The Coxeter system $(W, S)$ is called the \emph{type} of $X$.  In particular, a \emph{right-angled building} is a building whose Coxeter system is right-angled, and an  \emph{irreducible building} is a building whose Coxeter system is irreducible (the latter property admits also an intrinsic description).  The \emph{rank} of $X$ is by definition the rank of $(W,S)$, which is equal to the number of facets of any chamber in $X$.  The \emph{thickness} of $X$ is the tuple $(d_s)_{s \in S}$ where $d_s-1$ is the cardinality of the set of chambers containing the fundamental chamber's facet labelled by $s$ (so that a building is thick if and only if $d_s\geq 2$ for all $s \in S$).  The building is called \emph{locally finite} if all $d_s$, $s \in S$ are finite.  

Given a locally finite, thick building $X$, its group of simplicial automorphisms $\Aut(X)$ carries the topology of point-wise convergence on simplices, which turns $\Aut(X)$ into a totally disconnected, locally compact group.  A subgroup $G \leq \Aut(X)$ is called \emph{type preserving}, if it leaves invariant any labelling of facets of chambers in X obtained from the choice of a labelled fundamental chamber.  A subgroup $G \leq \Aut(X)$ is called \emph{strongly transitive} if it acts transitively on pairs $(A,C)$ where $A$ is an apartment of $X$ and $C$ is a chamber in $A$.  Given a closed, strongly transitive subgroup $G \leq \Aut(X)$, an \emph{Iwahori subgroup} of $G$, denoted below by $B$,  is the stabiliser of a chamber in $X$.  Every Iwahori subgroup is compact and open and it is unique up to conjugation by elements from $G$.  Iwahori subgroups play the role of the group $B$ in Tits' $(B,N)$-pairs, which can be used to construct buildings from groups.  If $(W, S)$ is the type of $X$, then the \emph{Bruhat decomposition} of a closed, strongly transitive and type preserving subgroup $G \leq \Aut(X)$ provides a bijection $W \to B \backslash G / B$ assigning to $w \in W$ its \emph{Bruhat cell} $BwB$. Note that this identification uses the fact that $W$ can be viewed as the group of type preserving automorphisms of an apartment of $X$. The \emph{Iwahori-Hecke algebra} of $G$ is the algebra $\CC[G,B]$ of $B$-biinvariant and compactly supported functions on $G$ equipped with the convolution product of $\contc(G)$.  We normalise the Haar measure of $G$ on some (or equivalently any) Iwahori subgroup.  A unitary representation $\pi$ of $G$ is called \emph{Iwahori-spherical} if it has a non-zero fixed vector under some Iwahori subgroup of $G$. 

The following identification of Hecke algebras goes back to Iwahori and Matsumoto \cite{iwahorimatsumoto65}.  We refer to \cite[Section 6.2]{garrett97} for a modern presentation, and for the convenience of the reader give a short proof, showing that the scaling factors are chosen correctly.
\begin{theorem}
  \label{thm:iwahori-hecke-algebra-is-generic}
  Let $X$ be a locally finite, thick building of type $(W,S)$ and $G \leq \Aut(X)$ be a closed, strongly transitive and type preserving subgroup.  Let $B \leq G$ be an Iwahori subgroup of $G$ and denote by $(d_s)_{s \in S}$ the thickness of $X$.  Let $q_s = d_s^{-1}$, for $s \in S$.  Then $\CC_\qm(W) \cong \CC[G, B]$ as *-algebras, with an isomorphism given by the map $T_w^{(\qm)} \mapsto (-1)^{|w|}\sqrt{q_w}\mathbb{1}_{BwB}$.
\end{theorem}
\begin{proof}
  In \cite[Section 6.2]{garrett97}, it is shown that $\CC[G,B]$ is presented as the algebra with relations
  \begin{gather*}
    \CC[G,B] = \bigl \langle (\mathbb{1}_{BwB})_{w \in W} \mid \mathbb{1}_{BsB} \mathbb{1}_{BwB} = \mathbb{1}_{BswB} \text{ if } |sw| > |w| , \quad
    \mathbb{1}_{BsB}^2 = (d_s - 1)\mathbb{1}_{BsB} + d_s
    \bigr \rangle \eqstop
  \end{gather*}
  Using the first relation, we see that for $s \neq t$ in $S$ the Artin relation
  \begin{gather*}
    \mathbb{1}_{BsB} \mathbb{1}_{BtB} \dotsm  = \mathbb{1}_{BtB} \mathbb{1}_{BsB} \dotsm \qquad (m_{st} \text{ factors})
  \end{gather*}
  is satisfied, where $m_{st}$ denotes the order of $st$ in $W$. The second relation implies that
  \begin{gather*}
    \left ( \frac{-1}{\sqrt{d_s}} \mathbb{1}_{BsB} \right )^2 = \frac{d_s-1}{d_s} \mathbb{1}_{BsB} + 1 = - p_s \sqrt{q_s}\mathbb{1}_{BsB} + 1
    \eqcomma
  \end{gather*}
  which is equivalent to the Hecke relation of $\CC_\qm(W)$.  So there is an algebra homomorphism $\vphi: \CC_\qm(W) \to \CC[G,B]$ satisfying $\vphi(T_s^{(\qm)}) = (-1)\sqrt{q_s}\mathbb{1}_{BsB}$ for all $s \in S$.  It  satisfies $\vphi(T_w^{(\qm)}) = (-1)^{|w|}\sqrt{q_w}\mathbb{1}_{BwB}$ for all $w \in W$ by the definition of $T_w^{(\qm)}$ and the defining relations of $\CC[G,B]$.  An inverse for $\vphi$ exists since $(-1)^{|w|}\sqrt{d_w} T_w^{(\qm)}$ satisfy the relations characterising $\CC[G,B]$.  So $\vphi$ is an algebra isomorphism.  We conclude by pointing out that it is a *-isomorphism thanks to the fact that $(T_w^{(\qm)})^* = T_{w^{-1}}^{(\qm)}$ and $\mathbb{1}_{BwB}^* = \mathbb{1}_{Bw^{-1}B}$.
\end{proof}

\section{Factoriality}
\label{sec:factoriality}

In this section, we will prove Theorem \ref{thm:factoriality} and Corollary \ref{cor:spherical-representations}.  Our strategy to prove Theorem \ref{thm:factoriality} combines an improved version of Garncarek's main analytic tool presented in Section \ref{sec:analytic} with a novel combinatoric approach to inductively reduce the complexity of right-angled Coxeter groups, developed in Section \ref{sec:coxeter-graphs}.  In the single parameter case, it provides an alternative way to obtain the results of \cite{garncarek16}.  Finally, in Section \ref{sec:proofs}, we prove Theorem \ref{thm:factoriality} and Corollary \ref{cor:spherical-representations}.

\subsection{Analytic aspects: determining coefficients on double cosets}
\label{sec:analytic}

The proof of the following lemma is standard.  Its single parameter (Hilbert space) version already appeared as \cite[Lemma 5.1]{garncarek16}; note that the notion of central vectors we are using is slightly more general than that of \cite{garncarek16}, but again the proof follows by a straightforward computation. %
\begin{lemma}
  \label{lem:commutation-relation}
  Let $(W, S)$ be a right-angled Coxeter system and $\qm \in (0,1]^S$ and let $r \in [1,\infty)$. Fix $s \in S$ and $\xi \in \lp(W)$.  Then $\xi$ is central for $T_s^{(\qm)}$ if and only if for all $w \in W$ satisfying $|sws| = |w| + 2$ we have
  \begin{align*}
    \xi_{sw} & = \xi_{ws} \\
    \xi_{sws} & = \xi_{w} + p_s \xi_{sw}
                \eqstop
  \end{align*}
\end{lemma}

\setlength{\parindent}{0pt}

The key new result in the analytic part of our treatment of the multiparameter case compared to \cite{garncarek16} is the following generalisation of \cite[Proposition 5.2]{garncarek16}.
\begin{lemma}
  \label{lem:restriction-double-coset}
  Let $(W,S)$ be a right-angled Coxeter system, let $\qm \in (0,1]^S$ and let $r \in [1,\infty)$.  Let $s,t \in S$ be non-commuting elements and write $D = \langle s, t \rangle$.  Assume that $\xi \in \lp(W)$ is central for $T_s^{(\qm)}$ and $T_t^{(\qm)}$.  Then on every non-degenerate double coset $DwD \subset W$ with shortest element $w$, the following relations hold.
  \begin{itemize}
  \item If $q_s = q_t$, then 
    \begin{gather*}
      \xi_{d w d'}
      =
      \xi_w \cdot \left ( \frac{q_{dwd'}}{q_w} \right)^{\frac{1}{2}}
      \qquad
      \text{for all }
      d,d' \in D
      \eqstop
    \end{gather*}
  \item If $q_s < q_t$, then there is $c \in \CC$ such that
    \begin{gather*}
      \xi_{d w d'}
      =
      \xi_w \cdot \bigl (
      c  \left ( \frac{q_{dwd'}}{q_w} \right)^{\frac{1}{2}}
      + (1 - c) \left ( \frac{q_{dwd', \vepsm}}{q_{w,\vepsm}} \right)^{\frac{1}{2}}
      \bigr )
      \qquad
      \text{for all }
      d,d' \in D
      \eqcomma
    \end{gather*}
    for all $\vepsm \in \{-1,1\}^S$ satisfying $\veps_s = 1$ and $\veps_t = -1$.
  \item If $q_s> q_t$, then there is $c \in \CC$ such that
    \begin{gather*}
      \xi_{d w d'}
      =
      \xi_w \cdot \bigl (
      c  \left ( \frac{q_{dwd'}}{q_w} \right)^{\frac{1}{2}}
      + (1 - c) \left ( \frac{q_{dwd', \vepsm}}{q_{w,\vepsm}} \right)^{\frac{1}{2}}
      \bigr )
      \qquad
      \text{for all }
      d,d' \in D
      \eqcomma
    \end{gather*}
    for all $\vepsm \in \{-1,1\}^S$ satisfying $\veps_s = -1$ and $\veps_t = 1$.
  \end{itemize}
\end{lemma}
\begin{proof}
	For simplicity we shall write $q_{v,1,-1}$ for $q_{v, \vepsm}$ whenever $v \in W$ and $\vepsm \in \{-1,1\}^S$ is such that $\veps_s = 1$ and $\veps_t = -1$, and similarly $q_{v,-1,1}$ for $q_{v, \vepsm}$ whenever $v \in W$ and $\vepsm \in \{-1,1\}^S$ is such that $\veps_s = -1$ and $\veps_t = 1$. 
	
  Since $DwD$ is non-degenerate, at least one of the elements $s,t$ does not commute with $w$.  We may without loss of generality assume that this is $t$.   Define a function $f: D \times D \to \CC$ by $f(d, d') = \xi_{d^{-1}wd'}$.  Then we have to find $a,b,c \in \CC$ satisfying $a + b + c = f(e,e)$ such that: $b = 0$ if $q_s \leq q_t$, $c = 0$ if $q_s \geq q_t$, and such that
  \begin{gather*}
    f(d,d')
    =
    a \left ( \frac{q_{d^{-1}wd'}}{q_w} \right)^{\frac{1}{2}}
    + b \left (\frac{q_{d^{-1}wd', -1,1}}{q_{w,-1,1}} \right)^{\frac{1}{2}}
    + c \left (\frac{q_{d^{-1}wd', 1,-1}}{q_{w,1,-1}} \right)^{\frac{1}{2}}
  \end{gather*}
  holds for all $d,d' \in D$.
  
  We will first collect several algebraic relations between values of $f$, obtained by applying Lemma \ref{lem:commutation-relation}.  Given $d,d' \in D$ that do not end in $t$ we have $|td^{-1}wd't| = |d^{-1}wd'| +2$, whence
  \begin{align}
    f(d, d't) & = f(dt, d') \label{eq:1}  \\
    f(d,d') & =  f(dt, d't) - p_t f(dt, d') \label{eq:2}
              \eqstop
  \end{align}
  If $s$ does not commute with $w$, we similarly find that for all $d,d' \in D$ that do not end in $s$, we have
  \begin{align}
    f(d, d's) & = f(ds, d')  \label{eq:3}  \\
    f(d,d') & = f(ds, d's)  - p_s f(ds, d') \label{eq:4}
                             \eqstop
  \end{align}
  If $s$ does commute with $w$, these statements do hold for $d,d' \in D$ that end in $t$ and additionally for the choice $d = e$ and $d' \in D$ that does not start or end with $s$, as in these cases $|sd^{-1}wd's| = |d^{-1}wd'| +2$ continues to hold.

  \begin{claim}
    \label{claim:1}
    The equality $f(d, d'ts) = f(dts, d')$ holds for all $d, d' \in D$ that do not end with $t$.
  \end{claim}
  \begin{proof}[Proof of the claim]
  We fix a pair of elements $d,d' \in D$ that do not end in $t$.  Write $d_n = d tstst \dotsm$ with $n$ letters attached, and similarly write $d'_n = d' tstst \dotsm$ with $n$ letters attached.  We can apply the relations \eqref{eq:2} and \eqref{eq:4} to obtain
  \begin{align*}
    f(d,d')
    & = f(dt,d't) - p_t f(dt, d') \\
    & = f(dts, d'ts) - p_s f(dts, d't) - p_t f(dt,d') \\
    & = f(d_{2n}, d'_{2n}) - \sum_{k = 0}^{n -1} p_s f(d_{2k}ts, d'_{2k}t) + p_tf(d_{2k}t, d'_{2k})
      \eqstop
  \end{align*}
  Since $f$ vanishes at infinity, we infer that the partial sums $\sum_{k = 0}^{n -1}  p_s f(d_{2k}ts, d'_{2k}t) + p_tf(d_{2k}t, d'_{2k})$ converge and we obtain the expression
  \begin{gather*}
    f(d,d')
    =
    - \sum_{k = 0}^\infty p_s f(d_{2k}ts, d'_{2k}t) + p_tf(d_{2k}t, d'_{2k})
    \eqstop
  \end{gather*}
  In order to prove our claim, it now suffices to see that $f(d(ts)^{k+1}t, d'(ts)^k) = f(d(ts)^kt, d'(ts)^{k+1})$ and $f(d(ts)^{k+2}, d'(ts)^kt) = f(d(ts)^{k+1}, d'(ts)^{k+1}t)$ hold for all $k \in \NN$.  These relations follow from equations \eqref{eq:1} and \eqref{eq:3}.
  \end{proof}

  \begin{claim}
    \label{claim:2}
    If $s$ and $w$ do not commute, then $f(d, d'st) = f(dst, d')$ holds for all $d, d' \in D$ that do not end with $s$.  If $s$ and $w$ commute, then $f(d, d'st) = f(dst, d')$ holds for all $d, d' \in D$ that end with $t$ as well as for $d = e$ under the additional condition that $d' \in W$ neither starts nor ends with $s$.
  \end{claim}
  \begin{proof}[Proof of the claim]
    The proof of this claim follows exactly the same way as the previous one, invoking equations \eqref{eq:2} and \eqref{eq:4} while respecting the special case where $s$ and $w$ commute.
  \end{proof}

  Let us define two functions $g_1,g_2: \NN \to \CC$ by the formulas
  \begin{align*}
    g_1(n) & = f((ts)^n, e) \\
    g_2(n) & = f((ts)^nt, e)
             \eqstop
  \end{align*}
  Then
  \begin{gather}
    \label{eq:rec1}
    g_1(n + 2) = g_1(n) + p_s g_2(n + 1) + p_t g_2(n) \qquad \text{for } n \in \NN \eqcomma
  \end{gather}
  since an application of equations \eqref{eq:2} and \eqref{eq:4}, and then of Claim \ref{claim:1} and equation \eqref{eq:1} shows that
  \begin{align*}
    f((ts)^n, e)
    & =
      f((ts)^nt,t) - p_tf((ts)^nt, e) \\
    & =
      f((ts)^nts,ts) - p_s f((ts)^nts,t) - p_t f((ts)^nt, e) \\
    & =
      f((ts)^{n+2},e) - p_s f((ts)^{n+1}t,e) -  p_tf((ts)^nt, e)
      \eqstop
  \end{align*}
  Similarly, the recurrence
  \begin{gather}
    \label{eq:rec2}
    g_2(n + 2) = g_2(n) + p_t g_1(n + 2) + p_s g_1(n + 1) \qquad \text{for } n \in \NN
  \end{gather}  
  follows from the calculation
  \begin{align*}
    f((ts)^nt, e)
    & =
      f((ts)^nts, s) - p_s f((ts)^nts, e) \\
    & =
      f((ts)^ntst, st) - p_t f((ts)^ntst, s) - p_s f((ts)^nts, e) \\
    & =
      f((ts)^{n+2}t, e) - p_t f((ts)^{n+2}, e) - p_s f((ts)^{n+1}, e)
      \eqstop
  \end{align*}
  The system of recurrence relations \eqref{eq:rec1}--\eqref{eq:rec2} associated with $g_1, g_2$ can be transformed into a linear system of four relations  with constant coefficients.  It follows that it has a system of four fundamental solutions.  Considering the solutions to the difference equation appearing in \cite[Proposition~5.2]{garncarek16} (and correcting the typo in there), one guesses the following solutions and verifies in elementary calculations that they indeed obey the recurrences.
  \begin{align*}
    \tilde g_1(n) & = q_s^{\frac{n}{2}} q_t^{\frac{n}{2}} \tag{First fundamental solution}\\
    \tilde g_2(n) & = q_s^{\frac{n}{2}} q_t^{\frac{n+1}{2}} \\[1em]
    \tilde g_1'(n) & = (-1)^n q_s^{-\frac{n}{2}} q_t^{\frac{n}{2}} \tag{Second fundamental solution}\\
    \tilde g_2'(n) & = (-1)^n q_s^{-\frac{n}{2}} q_t^{\frac{n+1}{2}} \\[1em]
    \tilde g_1''(n) & = (-1)^n q_s^{\frac{n}{2}} q_t^{-\frac{n}{2}} \tag{Third fundamental solution} \\
    \tilde g_2''(n) & = (-1)^{n+1} q_s^{\frac{n}{2}} q_t^{-\frac{n+1}{2}} \\[1em]
    \tilde g_1'''(n) & = q_s^{-\frac{n}{2}} q_t^{-\frac{n}{2}} \tag{Fourth fundamental solution}\\
    \tilde g_2'''(n) & = - q_s^{- \frac{n}{2}} q_t^{-\frac{n+1}{2}} \eqstop
  \end{align*}
  We exemplify the necessary calculations by checking equation \eqref{eq:rec1} for the first fundamental solution $(\tilde g_1, \tilde g_2)$. Dividing both sides of \eqref{eq:rec1} by $\tilde g_1(n)$, it remains to match the term $q_sq_t$ with 
  \begin{gather*}
    1 + p_s q_s^{\frac{1}{2}} q_t + p_t q_t^{\frac{1}{2}}
    =
    1 + (q_s - 1)q_t + (q_t - 1)
    =
    q_sq_t
    \eqstop
  \end{gather*}
  So the equation is indeed satisfied.
  
  Because $q_s, q_t \in (0,1]$, it follows that any solution to the recurrence relation that vanishes at infinity must be a linear combination of the first three solutions.  So we find $a,b,c \in \CC$ such that
  \begin{gather}
    \label{eq:fts}
    f((ts)^n, e) = g_1(n) = a q_s^{\frac{n}{2}} q_t^{\frac{n}{2}} + b (-1)^n q_s^{-\frac{n}{2}} q_t^{\frac{n}{2}} + c (-1)^n q_s^{\frac{n}{2}} q_t^{-\frac{n}{2}}
  \end{gather}
  and similarly
  \begin{gather*}
    f((ts)^nt, e) = g_2(n) =  a q_s^{\frac{n}{2}} q_t^{\frac{n + 1}{2}} + b (-1)^n q_s^{-\frac{n}{2}} q_t^{\frac{n + 1}{2}} + c (-1)^{n + 1} q_s^{\frac{n}{2}} q_t^{-\frac{n + 1}{2}}
    \eqstop
  \end{gather*}
  We note that $a + b + c = g_1(0) = f(e,e)$ and that $q_s < q_t$ implies $b = 0$ and $q_s > q_t$ implies $c = 0$, once again due to the behaviour when $n$ tends to infinity. When $q_s=q_t$ we deduce first from \eqref{eq:fts} that $b+c=0$ and then from the following equality that $b-c=0$.
  
  We will next investigate the pairs $((st)^n, e)$ and $((st)^n s, e)$ in a similar fashion, carefully taking care of possible commutation between $w$ and $s$.  We start with the following observation.
  \begin{claim}
    \label{claim:3}
    We have
    \begin{gather*}
      f((st)^n, e) = f((ts)^n, e)
    \end{gather*}
    for all $n \in \NN$. 
  \end{claim}
  \begin{proof}[Proof of the claim]
    The claim is clear for $n = 0$.  Inductively applying equations \eqref{eq:1} and \eqref{eq:3} combined with Claim \ref{claim:1} we find that for $n \geq 1$
    \begin{align*}
      f((st)^n,e) =
      & =
        f((st)^{n-1}s,t) \\
      & =
        f((st)^{n-1},ts) \\
      & =
        f(e,(ts)^n) \\
      & =
        f(ts,(ts)^{n-1}) \\
      & =
        f((ts)^n,e)
        \eqstop
    \end{align*}
  \end{proof}

  Let us define functions $h_1, h_2: \NN \to \CC$ by 
  \begin{align*}
    h_1(n) & = f((st)^n, e) \\
    h_2(n) & = f((st)^ns, e)
             \eqstop
  \end{align*}
  If $s$ and $w$ do not commute, then we can derive and solve recurrence relations as before and thereby find $a',b',c' \in \CC$ such that 
  \begin{align*}
    f((st)^n, e)
    & =
      a' q_s^{\frac{n}{2}} q_t^{\frac{n}{2}}
      + b' (-1)^n q_s^{-\frac{n}{2}} q_t^{\frac{n}{2}}
      + c' (-1)^n q_s^{\frac{n}{2}} q_t^{-\frac{n}{2}} \\
    f((st)^ns, e)
    & =
      a' q_s^{\frac{n}{2}} q_t^{\frac{n + 1}{2}}
      + b' (-1)^{n+1} q_s^{-\frac{n + 1}{2}} q_t^{\frac{n}{2}}
      + c' (-1)^n q_s^{\frac{n+1}{2}} q_t^{-\frac{n}{2}}
    \eqstop
  \end{align*}
  holds for all $n \in \NN$.  By Claim \ref{claim:3} it follows that $a' = a$, $b' = b$ and $c' = c$.  Indeed, if $q_s = q_t$, then $b= b' = c = c' = 0$.  Otherwise parameters can be matched by comparing values for $n \in \{0,1,2\}$.

  If $s$ and $w$ commute, we only obtain the relations (analogous to \eqref{eq:rec1} and \eqref{eq:rec2})
  \begin{align*}
    h_1(n + 2) & = h_1(n) + p_t h_2(n + 1) + p_s h_2(n) && \text{ for } n \in \NN \\
    h_2(n + 2) & = h_2(n) + p_s h_1(n + 2) + p_t h_1(n + 1) && \text{ for } n \in \NN_{\geq 1}
                 \eqstop
  \end{align*}
  The same method as before applies, showing that there are scalars $a'',b'',c'' \in \CC$ such that for $n \in \NN_{\geq 1}$ we have
  \begin{align*}
    h_1(n)
    & =
      a'' q_s^{\frac{n}{2}} q_t^{\frac{n}{2}}
      + b'' (-1)^n q_s^{-\frac{n}{2}} q_t^{\frac{n}{2}}
      + c'' (-1)^n q_s^{\frac{n}{2}} q_t^{-\frac{n}{2}} \\
    h_2(n)
    & =
      a'' q_s^{\frac{n}{2}} q_t^{\frac{n + 1}{2}}
      + b'' (-1)^n q_s^{-\frac{n}{2}} q_t^{\frac{n+1}{2}}
      + c'' (-1)^{n+1} q_s^{\frac{n}{2}} q_t^{-\frac{n+1}{2}}
             \eqstop
  \end{align*}
  Thanks to Claim \ref{claim:3}, we have $h_1(n) = g_1(n)$ for all $n \in \NN$.  In particular, the parameters $a'',b'',c''$ are determined by $a,b,c$ and a calculation shows that
  \begin{gather*}
    f((st)^ns,e)
    =
      a q_s^{\frac{n+1}{2}} q_t^{\frac{n}{2}}
      + b (-1)^{n+1} q_s^{-\frac{n+1}{2}} q_t^{\frac{n}{2}}
      + c (-1)^n q_s^{\frac{n+1}{2}} q_t^{-\frac{n}{2}}
      \qquad \text{for } n \in \NN_{\geq 1} \eqstop
  \end{gather*}

  It remains to find the value of $h_2(0) = f(s,e)$.  Plugging $n=0$ into the first recurrence relation for $h_1$ and $h_2$ and substituting $h_1(0)=f(e,e) = a+b+c$, we obtain
  \begin{gather*}
    a q_s q_t + b q_s^{-1}q_t + c q_s q_t^{-1}
    =
   a+b+c + p_t(a q_s q_t^{\frac{1}{2}} + b q_s^{-1} q_t^{\frac{1}{2}} - c q_s q_t^{-\frac{1}{2}}) + p_s f(s,e)
  \end{gather*}
  and thus
    \begin{gather*}
    f(s,e)
    =
    p_s^{-1} \bigl ( a (q_s q_t - 1 - p_t q_s q_t^{\frac{1}{2}})
    + b(q_s^{-1} q_t - 1 - p_t q_s^{-1} q_t^{\frac{1}{2}})
    + c(q_sq_t^{-1} - 1 + p_t q_sq_t^{-\frac{1}{2}}) \bigr)
    \eqstop
  \end{gather*}
  Simplifying each term on the right-hand side of the equation, we find
  \begin{gather*}
    f(s,e) = (a + c)q_s^{\frac{1}{2}} - b q_s^{-\frac{1}{2}}
    \eqstop
  \end{gather*}

  Summarising our results so far, we have found $a,b,c \in \CC$ such that
  \begin{align*}
    f(d, e)
    & =
      a q_{d}^{\frac{1}{2}} + b q_{d, -1, 1}^{\frac{1}{2}} + c q_{d, 1, -1}^{\frac{1}{2}} \\
    & =
      a \left ( \frac{q_{d^{-1}w}}{q_w} \right)^{\frac{1}{2}}
      + b \left (\frac{q_{d^{-1}w, -1,1}}{q_{w,-1,1}} \right)^{\frac{1}{2}}
      + c \left (\frac{q_{d^{-1}w, 1,-1}}{q_{w,1,-1}} \right)^{\frac{1}{2}}
  \end{align*}
  holds for all $d \in D$.  In the remainder of the proof, we will reduce all considerations about the values of $f$ to these cases.

  Since the transformations appearing in relations \eqref{eq:1} and \eqref{eq:3} as well as in Claims \ref{claim:1} and \ref{claim:2} do not affect the value of $q_{d^{-1}wd', \vepsm}$, we see that it remains to determine the values of $f$ on the following pairs of elements: 
  \begin{gather*}
    \begin{array}{llcll}
      ((st)^ns,s) & \text{for } n \in \NN \eqcomma && ((ts)^nt, t) & \text{for } n \in \NN, \\
      ((st)^n, t) & \text{for } n \in \NN_{\geq 1} \eqcomma && ((ts)^n, s) & \text{for } n \in \NN_{\geq 1}.
    \end{array}
  \end{gather*}
  Treating the exception first, we observe that if $s$ and $w$ commute, then $((st)^ns)^{-1}ws = (st)^ns w s  = (st)^n w$ and hence $f((st)^ns, s) = f((st)^n, e)$.  If $s$ and $w$ do not commute, we obtain
  \begin{align*}
    f((st)^ns, s)
    & = f((st)^n, e) + p_sf((st)^ns,e)
    \eqcomma
  \end{align*}
  and a calculation as before shows that also $f((st)^ns,s)$ satisfies the first equality stated in the lemma.  Similar calculations determine the value of $f$ on the pairs $((st)^n, t)$, $n \in \NN_{\geq 1}$ as well as $((ts)^nt,t)$, $n \in \NN$ and $((ts)^n, s)$, $n \in \NN_{\geq 1}$, thereby finishing the proof of the lemma.
\end{proof}

\setlength{\parindent}{1em}

\begin{remark}
  \label{rem:earlier-version-analytic}
   One should note that the `non-uniqueness' in the multiparameter case, namely a potential appearance of the linear combination in the above lemma, makes it far more difficult to combine the information obtained for different cosets; and it is precisely this point which neccessitates the intricate combinatorial argument given in the next subsection.
\end{remark}

\subsection{Combinatorial aspects: induction on Coxeter graphs}
\label{sec:coxeter-graphs}

In this section we will apply the result on coefficients on double cosets obtained in Section \ref{sec:analytic} in order prove Theorem \ref{thm:factoriality} inductively on the number of generators of the Coxeter system.  The key insight is that the statement of Theorem \ref{thm:hecke-central-vectors-hecke-ev} holds true for all right-angled Coxeter groups of rank one and of rank at least three, so that only dihedral groups $\ZZ/2 * \ZZ/2$ have to be avoided.  The underlying structure of the induction is then best described in terms of the Coxeter graph.  Given a right-angled Coxeter system $(W,S)$ we mean by its \emph{Coxeter graph} the unlabelled Coxeter diagram.  Concretely, the Coxeter graph is the simplicial graph with vertex set $S$ and edge set $\{\{s,t\} \mid [s,t] \neq e\}$.  The proof of Theorem \ref{thm:hecke-central-vectors-hecke-ev} chooses vertices and edges that can be removed from a Coxeter graph in order to appeal to an induction hypothesis.  Which vertices can be removed while keeping control over Hecke central vectors is depending on the chosen multiparameter $\qm \in (0,1]^S$.  %
The order of parameters allows to identify suitable vertices and edges to be removed from the Coxeter graph.  The necessary lemmas formalising this idea are presented in Section \ref{sec:reduction}.  Before, the induction base is treated in Section \ref{sec:base-cases}.  Then proof Theorem of \ref{thm:factoriality} will be presented in Section \ref{sec:proofs}.

\begin{remark}
  \label{rem:earlier-version-combinatorial}
  Note that in view of Remark \ref{rem:earlier-version-analytic}, compared to a single-parameter case a new strategy of proof had to be applied for Theorem \ref{thm:factoriality}, which replaces the use of Garncarek's combinatorial part \cite[Proposition 3.4]{garncarek16}.  Our strategy is novel even in the single parameter case treated by Garncarek, and provides an alternative -- though longer -- proof for this case.
\end{remark}

We first collect several short observations that will be used repeatedly in Sections \ref{sec:base-cases} and \ref{sec:reduction}.
\begin{lemma}
  \label{lem:adding-letters}
  Let $(W,S)$ be right-angled Coxeter system, let $\qm \in (0,1]^S$ and $\vepsm \in \{-1,1\}^S$.  Denote by $\eta_\veps$ the associated Hecke eigenvector, considered as a formal series.  If $w \in W$ and $s \in S$ satisfy $|sws| = |w| + 2$, then the following formulae hold.
    \begin{align*}
      (\eta_\vepsm)_{sws} & = (\eta_\vepsm)_w + p_s (\eta_\vepsm)_{ws} \quad \text{and}\\
      (\eta_\vepsm)_{sw} & = (\eta_\vepsm)_{ws}
                          \eqstop
    \end{align*}
    If $w \in W$ satisfies $|sw| > |w|$ and $\veps_s = 1$ holds, then
    \begin{gather*}
      (\eta_\vepsm)_{sw} = q_s^{\frac{1}{2}}(\eta_\vepsm)_w
      \eqstop
    \end{gather*}
\end{lemma}
\begin{proof}
  We only show the first identity.  Recall the definition $\eta_\vepsm = \sum_{w \in W} q_{w,\vepsm}^{\frac{1}{2}} \delta_w$.  So $(\eta_\vepsm)_{sws} = \veps_s q_{s,\vepsm} (\eta_\vepsm)_w$ and $(\eta_\vepsm)_{ws} = \veps_s q_s^{\veps_s \frac{1}{2}} (\eta_\vepsm)_w$ (recall that we are using the convention $q_{s, \veps}^{\frac{1}{2}} := \veps_s |q_{s,\veps_s}|^{\frac{1}{2}}$).  It hence suffices to show that $\veps_s q_{s,\vepsm} = 1 + \veps_s q_s^{\veps_s \frac{1}{2}} p_s$.  This identify follows directly when considering the values of $\veps_s$ separately, as was done in the proof of Proposition \ref{prop:hecke-eigenvectors}.  If $\veps_s = 1$, then
  \begin{gather*}
    1 + q_s^{\frac{1}{2}} p_s = 1 + q_s - 1 = q_{s , \vepsm}
    \eqstop
  \end{gather*}
  If $\veps_s = -1$, then
  \begin{gather*}
    1 - q_s^{-\frac{1}{2}} p_s = \frac{q_s}{q_s} - \frac{q_s - 1}{q_s} = \frac{1}{q_s} = - q_{s, \vepsm}
    \eqstop
  \end{gather*}
\end{proof}

The following lemma ensures the existence of Hecke eigenvectors in $\lp(W)$, which will allow us to deduce restrictions on the multiparameter.
\begin{lemma}
  \label{lem:hecke-ev-exists}
  Let $(W,S)$ be a right-angled Coxeter system.  For $\vepsm \in \{-1,1\}^S$ let $\eta_\vepsm$ be the associated Hecke eigenvector, considered a priori as a formal series.  Let $r \in [1,\infty)$ and let $c_\vepsm$, $\vepsm \in \{-1,1\}^S$ be complex coefficients such that $\sum_\vepsm c_\vepsm \eta_\vepsm \in \lp(W)$.  If $c_\vepsm \neq 0$ then $\eta_\vepsm  \in \lp(W)$.
\end{lemma}
\begin{proof}
  Recall from Proposition \ref{prop:hecke-eigenvectors} that the weight of $\eta_\vepsm$ is given by
  \begin{gather*}
    T_s^{(\qm)} \eta_\vepsm = \veps_s q_s^{\veps_s \frac{1}{2}} \eta_\vepsm, \;\;\; s \in S
    \eqstop
  \end{gather*}
  This implies that
  \begin{gather*}
    (T_s^{(\qm)} + q_s^{-\frac{1}{2}}) \eta_\vepsm =
    \begin{cases}
      (q_s^{\frac{1}{2}} + q_s^{-\frac{1}{2}}) \eta_\vepsm & \text{ if } \veps_s = 1 \\
      0 & \text{ if } \veps_s = -1
      \eqstop
    \end{cases}
  \end{gather*}
  and
  \begin{gather*}
    (T_s^{(\qm)} - q_s^{\frac{1}{2}}) \eta_\vepsm =
    \begin{cases}
      0 & \text{ if } \veps_s = 1 \\
      -(q_s^{\frac{1}{2}} + q_s^{-\frac{1}{2}}) \eta_\vepsm & \text{ if } \veps_s = -1
      \eqstop
    \end{cases}
  \end{gather*}
  Hence for fixed $\vepsm$, we find that
  \begin{gather*}
    \prod_{s \in S} (T_s^{(\qm)} + \veps_s q_s^{- \veps_s \frac{1}{2}} ) (\sum_{\vepsm'} c_{\vepsm'} \eta_{\vepsm'}) = c c_\vepsm \eta_\vepsm
  \end{gather*}
  for some non-zero scalar $c$.  So if $c_\vepsm \neq 0$, then $\eta_\vepsm \in \lp(W)$.
\end{proof}

\subsubsection{Base cases}
\label{sec:base-cases}

The base case of our induction describes Hecke central vectors for irreducible right-angled Coxeter groups of rank three, which are precisely $\ZZ/2^{\oplus 2} * \ZZ/2$ and $\ZZ/2^{*3}$.  The following lemmas treat these cases, depending on the how a given multiparameter orders the generators.

\begin{lemma}
  \label{lem:rank-three-straight-order}
  Let $(W, \{s,t,u\})$ be an irreducible right-angled Coxeter group such that $[s,t] \neq e$ and $[t,u] \neq e$.  Let $\qm \in (0,1]^{\{s,t,u\}}$ satisfy $q_s \leq q_t \leq q_u$.  Then every Hecke central vector for the parameter $\qm$ in $\lp(W)$ is a linear combination of Hecke eigenvectors and $\delta_e$.
\end{lemma}
\begin{proof}
  Let $\xi \in \lp(W)$ be a Hecke central vector.  Consider Hecke eigenvectors as formal series, which are not necessarily elements of $\lp(W)$.  We start by finding coefficients $c_\veps$, $\veps \in \{-1,1\}$ and $c_e$ such that
  \begin{gather*}
    \xi_w = (\sum_{\veps \in \{-1,1\}} c_\veps \eta_\tvepsm + c_e \delta_e)_w
      \quad \text{ for all } w \in \{e,u,s\}
      \eqcomma
    \end{gather*}
    where $\eta_\tvepsm$ denotes the Hecke eigenvector for $(W, \{s,t,u\} ,\qm)$ associated with the sign $\tilde \vepsm \in \{-1,1\}^{\{s,t,u\}}$ satisfying $\tilde \veps_s = \tilde \veps_t = 1$ and $\tilde \veps_u = \veps$.  Existence of these coefficients is clear, as the vectors $\eta_\tvepsm$, $\veps \in \{-1,1\}$ are linearly independent when restricted to $\{e,u,s\}$. %
  Write $\eta = \sum_{\veps \in \{-1,1\}} c_\veps \eta_{\tilde \veps} + c_e \delta_e$ still considered as a formal series and consider the formal series $\zeta = \eta - \xi$.  We show step by step that $ \zeta_w=0$ for all $w \in W$; note that although $\zeta$ is in principle a formal series which need not define an element of $\ell^r(W)$ we can apply to it the linear recurrence formulas as long as we know they are valid both for a concrete formal series $\eta_{\tilde \veps}$ and central vectors in $\ell^r(W)$, appealing to Lemmas~\ref{lem:commutation-relation}, \ref{lem:restriction-double-coset} and \ref{lem:adding-letters}.  We shall often use this fact below without any further comment.

    We first consider words not containing $t$ and distinguish two cases.  If $[s,u] = e$, then applying Lemma \ref{lem:restriction-double-coset} to the non-degenerate double coset $\lang s,t \rang su \lang s,t \rang$ and making use of the fact that $q_s \leq q_t$, we find that
    \begin{gather*}
      \zeta_{su} = q_s^{\frac{1}{2}} \zeta_u = 0%
      \eqcomma
    \end{gather*}
    where the first identity follows from Lemma \ref{lem:restriction-double-coset} applied to $\xi$ and~\ref{lem:adding-letters} applied to $\eta$ making use of the fact that $\tilde \veps_s = 1$.  This shows that $ \zeta_w=0$ for all $w \in \lang s,u \rang = \{e,s,u,su\}$. Considering the second case $[s,u] \neq e$, we prove that $\zeta_w=0$ for all $w \in \lang s,u \rang$ by induction on the length $|w|$.  The statement is clear for $|w| \leq 1$.  Let $w \in \lang s,u \rang$ be of length at least two.  If $w$ starts or ends with $s$, then $w$ contains some $u$, so that $\lang s,t \rang w \lang s,t \rang$ is non-degenerate. Considering the situation where $w$ starts with $s$, a similar argument as above combined with the induction hypothesis implies that
    \begin{gather*}
      \zeta_w = q_s^{\frac{1}{2}} \zeta_{sw} = 0 %
      \eqstop
    \end{gather*}
    The analogue calculation applies in the case $w$ ends with $s$.  If $w$ neither  starts nor ends with $s$, then it starts and ends with $u$, so that Lemmas \ref{lem:commutation-relation} and \ref{lem:adding-letters} say that
    \begin{gather*}
      \zeta_w = \zeta_{uwu} + p_u \zeta_{uw}
      \eqstop
    \end{gather*}
    The induction hypothesis thus shows that $ \zeta_w=0$.

    Let us next consider $w \in W$ containing some instance of $t$.  We assume that $w$ starts with $t$ and that $ \zeta_x=0$ holds for all $x \in W$ satisfying $\ell_t(x) < \ell_t(w)$.  If $w$ contains $s$, then $\lang t,u \rang w \lang t,u \rang$ is non-degenerate, so that
    \begin{gather*}
      \zeta_w = q_t^{\frac{1}{2}} \zeta_{tw} = 0 %
      \eqcomma
    \end{gather*}
    where the first equality makes use of the fact that the fact that $q_t \leq q_u$ and that $\tilde \veps_t = 1$ and the second makes use of the induction assumption. %
    If $w$ does not contain $s$, that is $w \in \lang u,t \rang$, then $\lang s,t \rang w \lang s,t \rang$ is non-degenerate.  The last part implies $ \zeta_{ws}=0$ and we find that
    \begin{gather*}
      \zeta_w = q_s^{-\frac{1}{2}} \zeta_{ws} = 0 %
    \end{gather*}
    where the first equality makes use of the fact that $\tilde \veps_s = 1$.

    We now consider words $w \in W$ such that $\ell_t(w) \geq 1$ and such that $\pos_t(w) \geq 2$.  We assume that $\zeta_x=0$ for all $x \in W$ such that $\ell_t(x) = \ell_t(w)$ and $\pos_t(x) < \pos_t(w)$.  There is $v \in \{s,u\}$ such that $w$ starts with $v$.  If $w$ also ends with $v$, then Lemma \ref{lem:commutation-relation}, the induction hypothesis and Lemma \ref{lem:adding-letters} imply that
    \begin{gather*}
      \zeta_w = \zeta_{vwv} - p_v \zeta_{vw} = 0 %
      \eqstop
    \end{gather*}
    If $w$ does not end with $v$, then we make use of the fact that $\pos_t(vwv) < \pos_t(w)$ and find that
    \begin{gather*}
      \zeta_w = \zeta_{vwv} = 0 %
      \eqstop
    \end{gather*}
    This concludes the proof of the lemma.    
\end{proof}

The next two lemmas treat specifically the case of $\ZZ/2^{\oplus 2} * \ZZ/2$. Each of them considers another constellation of the values of a multiparameter.
\begin{lemma}
  \label{lem:rank-three-free-generator-largest}
  Let $(W, \{s,t,u\})$ be the irreducible right-angled Coxeter system satisfying $[s,t] \neq e$, $[t,u] \neq e$ and $[s,u] = e$.  Let $\qm \in (0,1]^{\{s,t,u\}}$ satisfy $q_s, q_u \leq q_t$.  Then every Hecke central vector for the parameter $\qm$ in $\lp(W)$ is a linear combination of Hecke eigenvectors and $\delta_e$.  
\end{lemma}
\begin{proof}
  Let $\xi \in \lp(W)$ be a Hecke central vector.  Consider Hecke eigenvectors as formal series, which are not necessarily elements of $\lp(W)$.  We start by finding coefficients $c_\veps$, $\veps \in \{-1,1\}$ and $c_e$ such that
  \begin{gather*}
    \xi_w = (\sum_{\veps \in \{-1,1\}} c_\veps \eta_\tvepsm + c_e \delta_e)_w
      \quad \text{ for all } w \in \{e,s,t\}
      \eqcomma
    \end{gather*}
    where $\eta_\tvepsm$ denotes the Hecke eigenvector for $(W,\{s,t,u\},\qm)$ associated with the sign $\tilde \vepsm \in \{-1,1\}^{\{s,t,u\}}$ satisfying $\tilde \veps_s = 1$, $\tilde \veps_t = \veps$ and $\tilde \veps_u = 1$.  Existence of these coefficients is clear, as the vectors $\eta_\tvepsm$, $\veps \in \{-1,1\}$ are linearly independent when restricted to $\{e,s,t\}$. %
 Write $\eta = \sum_{\veps \in \{-1,1\}} c_\veps \eta_\tvepsm + c_e \delta_e$ and consider the formal series $\zeta= \xi - \eta$.  Then $ \zeta_w=0$ for all $w \in \{e,s,t\}$.  We will show that $\zeta_w=0$ for all $w \in W$.

    First observe that
    \begin{gather*}
      \zeta_{su} = q_u^{\frac{1}{2}} \zeta_s =  0 %
      \eqcomma
    \end{gather*}
    by an application of Lemma \ref{lem:restriction-double-coset} to the non-degenerate double coset $\lang u,t \rang s \lang u,t \rang$ and Lemma \ref{lem:adding-letters} making use of the fact that $\tilde \epsilon_u = 1$.  Similarly, we find that
    \begin{gather*}
      \zeta_u = q_s^{-\frac{1}{2}} \zeta_{su} = 0 %
      \eqstop
    \end{gather*}

    Next consider  $w \in \lang s,t \rang$ of length at least $2$ and assume that $\zeta_x = 0$ for all $x \in \lang s,t \rang$ such that $|x| < |w|$.
    If $w$ starts with $s$, we first make use of the fact that $q_u \leq q_t$ and then of the fact that $q_s \leq q_t$ to find that
    \begin{gather*}
      \zeta_w = q_u^{-\frac{1}{2}} \zeta_{wu} = q_s^{\frac{1}{2}} q_u^{-\frac{1}{2}} \zeta_{swu} 
      \eqstop
    \end{gather*}
    We can also use the double cosets $\langle u,t \rangle w \langle u,t \rangle$ and $\langle s,t \rangle uwu \langle s,t \rangle$ to find that
    \begin{gather*}
      \zeta_w = q_s^{\frac{1}{2}} q_u^{-1} \zeta_{suwu} = q_s^{\frac{1}{2}} q_u^{-1} \zeta_{uswu}
      \eqstop
    \end{gather*}
    This allows us to employ Lemmas \ref{lem:commutation-relation} and \ref{lem:adding-letters} to the extent that
    \begin{gather*}
      \zeta_{sw}
      =
      \zeta_{uswu} - p_u \zeta_{swu}
      =
      q_s^{-\frac{1}{2}} q_u \zeta_w - p_u q_s^{-\frac{1}{2}} q_u^{\frac{1}{2}} \zeta_w
      =
      q_s^{-\frac{1}{2}} (q_u - p_u q_u^{\frac{1}{2}}) \zeta_w
      =
      q_s^{-\frac{1}{2}} \zeta_w
      \eqstop
    \end{gather*}
    Hence we find that
    \begin{gather*}
      \zeta_w = q_s^{\frac{1}{2}} \zeta_{sw} = 0 %
      \eqcomma
    \end{gather*}
making use of the induction hypothesis.  We also find that $ \zeta_w=0$ if $w$ starts with $t$, appealing to Lemmas~\ref{lem:commutation-relation} and \ref{lem:adding-letters}.  We have thus found that $ \zeta_w=0$ for all $w \in \lang s,t \rang$.

    It remains to show that $ \zeta_w=0$ for all $w \in W$ that contain $u$ and $t$. Note that if $w \in \lang u, t \rang$ we can argue as above, as the roles of $s$ and $u$ are symmetric. Fix then $w \in W$ which contains $u$ and $t$ but does not belong to $\lang u, t \rang$,  and assume further that $w$ starts with $u$ and that $ \zeta_x=0$ for all words $x \in W$ such that $\ell_u(x) < \ell_u(w)$.  Then Lemmas \ref{lem:restriction-double-coset} and \ref{lem:adding-letters} show that
    \begin{gather*}
      \zeta_w = q_u^{\frac{1}{2}} \zeta_{uw} = 0%
      \eqcomma
    \end{gather*}
    since $\tilde \veps_u = 1$.  Next take $w \in W$ containing $u$ and satisfying $\pos_u(w) \geq 2$.  Assume that $ \zeta_x=0$ for all $x \in W$ such that $\ell_u(x) = \ell_u(w)$ and $\pos_u(x) < \pos_u(w)$.  Note that $w$ does not start with $u$.  Assume first that $w$ starts with $t$.  If $w$ also ends with $t$, then Lemmas \ref{lem:commutation-relation} and \ref{lem:adding-letters} apply, saying that
    \begin{gather*}
      \zeta_w = \zeta_{twt} - p_t \zeta_{tw} = 0 %
      \eqstop
    \end{gather*}
    If $w$ does not end with $t$, then the same lemmas imply that
    \begin{gather*}
      \zeta_w = \zeta_{twt} = 0 %
      \eqstop
    \end{gather*}
    The case of $w$ starting with $s$ can be dealt with in a similar way.  This finishes the proof of the lemma.
\end{proof}

\begin{lemma}
  \label{lem:rank-three-free-generator-smallest}
  Let $(W, \{s,t,u\})$ be the irreducible right-angled Coxeter system satisfying $[s,t] \neq e$, $[t,u] \neq e$ and $[s,u] = e$.  Let $\qm \in (0,1]^{\{s,t,u\}}$ satisfy $q_s, q_u \geq q_t$.  Then every Hecke central vector for the parameter $\qm$ in $\lp(W)$ is a linear combination of Hecke eigenvectors and $\delta_e$.  
\end{lemma}
\begin{proof}
  Let $\xi \in \lp(W)$ be a Hecke central vector.  Once again consider Hecke eigenvectors as formal series, which are not necessarily elements of $\lp(W)$, and start by finding coefficients $c_\veps$, $\veps \in \{-1,1\}^{\{s,u\}}$ and $c_e$ such that
  \begin{gather*}
    \xi_w = (\sum_{\vepsm \in \{-1,1\}^{\{s,u\}}} c_\vepsm \eta_\tvepsm + c_e \delta_e)_w
      \quad \text{ for all } w \in \{e,s,u,su,t\}
      \eqcomma
    \end{gather*}
    where $\eta_\tvepsm$ denotes the Hecke eigenvector for $(W,\{s,t,u\},\qm)$ associated with the sign $\tilde \vepsm \in \{-1,1\}^{\{s,t,u\}}$ satisfying $\tilde \veps_s = \veps_s$, $\tilde \veps_t = 1$ and $\tilde \veps_u = \veps_u$.  Existence of these coefficients is clear, as the vectors $\eta_\tvepsm$, $\tvepsm \in \{-1,1\}^{\{s,u\}}$ are linearly independent when restricted to $\{e,s,u,su,t\}$.
    Write $\eta = \sum_{\vepsm \in \{-1,1\}^{\{s,u\}}} c_\vepsm \eta_\tvepsm + c_e \delta_e$ and consider the formal series $\zeta = \eta - \xi$.  We will show that $ \zeta_w=0$ for all $w \in W$.

    By choice of the coefficients $c_\vepsm$, we have $\zeta_w=0$ for all $w \in W$ that do not contain $t$ as well as for $w = t$.  Let $w \in W$ be a word of length at least two that contains $t$ and starts with $t$.  We assume that $ \zeta_x=0$ holds for all $x \in W$ such that $\ell_t(x) < \ell_t(w)$.  Then at least one of the double cosets $\lang s,t \rang w \lang s,t \rang$ or $\lang t,u \rang w \lang t,u \rang$ is non-degenerate, so that we can apply Lemmas \ref{lem:restriction-double-coset} and \ref{lem:adding-letters} to find that
    \begin{gather*}
      \zeta_w = q_t^{\frac{1}{2}} \zeta_{tw} =0%
      \eqstop
    \end{gather*}
    Let $w \in W$ be a word of length at least two that contains $t$ and such that $\pos_t(w) \geq 2$.  We assume that $\zeta_x=0$ holds for all $x \in W$ such that $\ell_t(x) = \ell_t(w)$ and $\pos_t(x) < \pos_t(w)$.  Let $v \in \{s,u\}$ be an initial letter of $w$.  If $v$ is also a terminal letter of $w$, Lemmas \ref{lem:commutation-relation} and \ref{lem:adding-letters} say that
    \begin{gather*}
      \zeta_w
      =
      \zeta_{vwv} - p_v \zeta_{vw}
      = 0
      \eqstop
    \end{gather*}
    Similarly, if $v$ is not a terminal letter of $w$, then
    \begin{gather*}
      \zeta_w
      =
      \zeta_{vwv}
      = 0
      \eqstop
    \end{gather*}
    This finishes the proof of the lemma.
\end{proof}

\begin{remark}
  \label{rem:dykema}
  It is worth noting that in the base cases above the description of Hecke central vectors in $\ltwo(W)$, which corresponds to computing centres of the associated Hecke von Neumann algebras, can be obtained by means of Dykema's work \cite[Theorem 2.3]{dykema93-hyperfinite} and %
  \cite[Proposition 2.4]{dykema93-hyperfinite}, which identifies the structure of von Neumann algebras arising as (iterated) free products of finite-dimensional algebras.
\end{remark}

\subsubsection{Reduction}
\label{sec:reduction}

In this section we prove two lemmas that will allow us to manipulate Coxeter graphs in Section~\ref{sec:proofs}.  Lemma \ref{lem:remove-arrows-head} will allow us to remove specific vertices depending on the multiparameter under consideration, while its combination with Lemma \ref{lem:remove-edge} will allow us to remove leaves for arbitrary multiparameters.

\begin{lemma}
  \label{lem:remove-arrows-head}
  Let $(W,S)$ be an irreducible right-angled Coxeter system, let $\qm \in (0,1]^S$ and let $r \in [1,\infty)$.  Assume that there are pairwise different letters $s,t,u \in S$ such that
  \begin{itemize}
  \item $[s,t] \neq e$ and $[u,t] \neq e$; %
  \item $q_s \leq q_t$;
  \item every Hecke central vector of $W|_{S \setminus s}$ in $\lp(W|_{S \setminus s})$ is a linear combination of Hecke eigenvectors and $\delta_e$.
  \end{itemize}
  Then every Hecke central vector for the multiparameter $\qm$ in $\lp(W)$ is a linear combination of Hecke eigenvectors and $\delta_e$.
\end{lemma}
\begin{proof}
  Let $\xi \in \lp(W)$ be a Hecke central vector.  Using Proposition \ref{prop:hecke-eigenvectors} and the assumption on $W|_{S \setminus s}$ we find coefficients $c_\vepsm \in \CC$, $\vepsm \in \{-1,1\}^{S \setminus s}$ and $c_e \in \CC$ such that
  \begin{gather*}
    \xi_w = (\sum_\vepsm c_\vepsm \eta_\tvepsm  + c_e \delta_e)_w
  \end{gather*}
  for all $w \in W|_{S \setminus s}$, where $\eta_\tvepsm$ denotes the Hecke eigenvector for $W$ associated with the sign $\tvepsm$  such that $\tilde \veps_t = \veps_t$ for $t \neq s$ and $\tilde \veps_s = 1$.  Write $\eta = \sum_\vepsm c_\vepsm \eta_\tvepsm  + c_e \delta_e$, where the right-hand side is considered a formal series, and consider the formal series $\zeta = \xi - \eta$. We will show that $\zeta_w=0$ for all $w \in W$.
  
  Let $w \in W$ be a word starting with $s$.  Assume that $ \zeta_x=0$ for all $x \in W$ satisfying $\ell_s(x) < \ell_s(w)$.  We consider different cases.

  \noindent \textbf{Case 1}. The word $w$ contains $u$.  In this case the double coset $\lang s,t \rang w \lang s,t \rang$ is non-degenerate so that Lemmas \ref{lem:restriction-double-coset} and \ref{lem:adding-letters} can be applied in view of $\tilde \veps_s = 1$ and combined with the induction hypothesis imply that
  \begin{gather*}
    \zeta_w = q_s^{\frac{1}{2}} \zeta_{sw} = 0 %
    \eqstop
  \end{gather*}

  \noindent \textbf{Case 2}. The word $w$ does not contain $u$, but it does contain $t$.  In this case we have $|uwu| = |w| + 2$, so that Lemmas \ref{lem:commutation-relation} and \ref{lem:adding-letters} say
  \begin{gather*}
    \zeta_w = \zeta_{uwu} - p_u \zeta_{wu}
    \eqstop
  \end{gather*}
If $[s,u] = e$,  both $uwu$ and $wu$ start with $s$, so that we can conclude thanks to Case 1.
On the other hand if  $[s,u] \neq e$, then still $wu$ starts with $s$, but $uwu$ does not.  The double coset $\lang s,u \rang uwu \lang s,u \rang$ is non-degenerate and since $w$ contains $t$, the number of occurences of $s$ and $u$ in $uwu$ and $susu(uwu)$ is the same.  So Lemmas \ref{lem:restriction-double-coset} and \ref{lem:adding-letters} apply and say that
 \begin{gather*}
   \zeta_{uwu}
   = \zeta_{susu (uwu)} = \zeta_{suswu}
    \eqstop
  \end{gather*}
  We can thus invoke Case 1 again to infer $ \zeta_w=0$.

  \noindent \textbf{Case 3}. The word $w$ does not contain $u$ nor $t$.  Since $w$ contains $s$, we observe that $|twt| = |w| + 2$.  So Lemma \ref{lem:commutation-relation} says that
  \begin{gather*}
    \zeta_w = \zeta_{twt} - p_t \zeta_{wt}
    \eqstop
  \end{gather*}
  Directly, Case 2 implies that $\zeta_{wt} = 0$. If we can also show that $ \zeta_{twt}=0$ then $\zeta_w=0$ follows.  In order to show $ \zeta_{twt}=0$, we notice that $\lang t,u \rang twt \lang t,u \rang$ is non-degenerate, since $w$ contains $s$.  If $q_t \leq q_u$,  then Lemma \ref{lem:hecke-ev-exists} applied to $W|_{S \setminus s}$ implies that $\veps_t = 1$ for all $\veps \in \{-1,1\}^{S \setminus s}$ such that $c_\veps \neq 0$.  So a combination of Lemma \ref{lem:restriction-double-coset} and Lemma \ref{lem:adding-letters} shows that
  \begin{gather*}
    \zeta_{twt} = q_t^{\frac{1}{2}} \zeta_{wt} = 0 %
    \eqstop
  \end{gather*}
  If $q_u \leq q_t$ then an appeal to Lemmas \ref{lem:commutation-relation}, \ref{lem:restriction-double-coset} and \ref{lem:adding-letters} shows that
  \begin{gather*}
    \zeta_{twt} = q_u^{-\frac{1}{2}} \zeta_{twtu} = q_u^{-\frac{1}{2}} \zeta_{wtut} = 0 %
    \eqstop
  \end{gather*}
  This finishes the considerations of Case 3.

  Let us now consider $w \in W$ that contains $s$, but does not start with $s$.  We assume that $\zeta_x =0$ for all $x \in W$ such that $\ell_s(x) = \ell_s(w)$ and $\pos_s(x) < \pos_s(w)$.  There is $v \in S$ such that $\pos_s(vw) < \pos_s(w)$. In particular, $w$ starts with $v$.  In case $w$ also ends with $v$, Lemmas \ref{lem:commutation-relation} and \ref{lem:adding-letters} imply that
  \begin{gather*}
    \zeta_w = \zeta_{vwv} - p_v \zeta_{vw} = 0 %
    \eqstop
  \end{gather*}
  In case $w$ does not end with $v$, the same lemmas imply that
  \begin{gather*}
    \zeta_w = \zeta_{vwv} = 0 %
    \eqstop
  \end{gather*}
\end{proof}

\begin{lemma}
  \label{lem:remove-edge}
  Let $(W,S)$ be an irreducible right-angled Coxeter system, let $\qm \in (0,1]^S$ and let $r \in [1,\infty)$.  Assume that there are pairwise different letters $s,t,u \in S$ such that
  \begin{itemize}
  \item $[s,t] \neq e$ and $[u,t] \neq e$;
  \item $\forall r \in S \setminus t: \, [r,s] = e$;
  \item $q_s \geq q_t$;
  \item every Hecke central vector of $W|_{S \setminus s}$ in $\lp(W|_{S \setminus s})$ is a linear combination of Hecke eigenvectors and $\delta_e$.
  \end{itemize}
  Then every Hecke central vector for the multiparameter $\qm$ in $\lp(W)$ is a linear combination of Hecke eigenvectors and $\delta_e$.
\end{lemma}
\begin{proof}
  Write $H = W/[s,t]$, so that $H$ is a reducible right-angled Coxeter group, decomposing as a product $ W|_{S \setminus s} \oplus \ZZ/2$.  Since the space $\lp(\ZZ/2)$ is spanned by Hecke eigenvectors, the assumption implies that every Hecke central vector of $H$ in $\lp(H) \cong \lp(W|_{S \setminus s}) \ot \lp(\ZZ/2)$ is a linear combination of Hecke eigenvectors, $\delta_e$ and $\delta_s$.  Consider the set of all words containing at most one instance of $s$, which, if it exists, must necessarily be located at the beginning of the word.
  \begin{gather*}
    G = \{ w\in W \mid \ell_s(w) \leq 1 \text{ and } \pos_s(w) \leq 1 \}
    \eqstop
  \end{gather*}
  The quotient map $W \to H$ restricts to a bijection $\vphi: G \to H$.

  Let $\xi \in \lp(W)$ be Hecke central.  We claim that $\vphi_*(\xi|_G) \in \lp(H)$ is Hecke central too.  By Lemma~\ref{lem:commutation-relation} we have to check that
  \begin{align*}
    \vphi_*(\xi|_G)_{rwr} & = \vphi_*(\xi|_G)_w + p_r \vphi_*(\xi|_G)_{wr} \\
    \vphi_*(\xi|_G)_{rw} & = \vphi_*(\xi|_G)_{wr}
  \end{align*}
  hold for all $w \in H$ and $r \in S$ such that $|rwr| = |w| + 2$.  We first notice if $r \in S$ satisfies the last equality of word lengths for some $w \in H$, then $r \neq s$.  Further, if $r \neq t$, then the assumption $[s,r] = e$ implies that $\vphi^{-1}(rw) = r\vphi^{-1}(w)$ and thus
  \begin{align*}
    \vphi_*(\xi|_G)_{rwr} & = \xi_{\vphi^{-1}(rwr)} = \xi_{r\vphi^{-1}(w)r} \\
    \vphi_*(\xi|_G)_{wr} & = \xi_{\vphi^{-1}(w)r} \\
    \vphi_*(\xi|_G)_{rw} & = \xi_{r\vphi^{-1}(w)}
                           \eqstop
  \end{align*}
  Since $|r \vphi^{-1}(w) r| - |\vphi^{-1}(w)| = |rwr| - |w|$ (as can be checked using the two possible forms of $w \in H$), Lemma \ref{lem:commutation-relation} applies inside $W$.  So it remains to check the conditions of Lemma \ref{lem:commutation-relation}  for $w \in H$ and $r = t$.  If $w \in H$ does not contain $s$, then $\vphi^{-1}(tw) = t \vphi^{-1}(w)$ and the previous argument applies.  So consider $w \in H$ containing $s$ and assume that $|twt| = |w| + 2$ in $H$.  This implies that $w$ contains some letter $v \neq s$ such that $[t,v] \neq e$.  So the double coset $\lang s,t \rang \vphi^{-1}(w) \lang s,t \rang$ is non-degenerate in $W$.  Further, we observe that the letter $t$ is not cancelled in the expression $sts\vphi^{-1}(w)$, so that there is an equal number of occurrences of $s$ and $t$ when comparing $sts\vphi^{-1}(w)$ with $\vphi^{-1}(tw)$.  Hence, Lemma \ref{lem:restriction-double-coset} implies that
  \begin{gather*}
    \xi_{\vphi^{-1}(tw)} = \xi_{sts\vphi^{-1}(w)} = \xi_{t \vphi^{-1}(w)}
    \eqstop
  \end{gather*}
  We can thus infer that
  \begin{align*}
    \vphi_*(\xi|_G)_{tw}
    & =
      \xi_{\vphi^{-1}(tw)} =
      \xi_{t\vphi^{-1}(w)} =
      \xi_{\vphi^{-1}(w)t} =
      \xi_{\vphi^{-1}(wt)} =
      \vphi_*(\xi|_G)_{wt}
      \eqstop      
  \end{align*}
  Similarly, we obtain
  \begin{align*}
    \vphi_*(\xi|_G)_{twt}
    & =
      \xi_{\vphi^{-1}(twt)}  =
      \xi_{tsts\vphi^{-1}(tw)t} =
      \xi_{t \vphi^{-1}(w)t}  =
      \xi_{\vphi^{-1}(w)} + p_t \xi_{\vphi^{-1}(w)t}  =
      \xi_{\vphi^{-1}(w)} + p_t \xi_{\vphi^{-1}(wt)} \\
    & =
      \vphi_*(\xi|_G)_w + p_t \vphi_*(\xi|_G)_{wt}
      \eqstop
  \end{align*}
  So indeed $\vphi_*(\xi|_G)$ is a Hecke central vector.  As argued in the first paragraph, it is thus the linear combination of Hecke eigenvectors.  Thus there are coefficients $c_\vepsm$, $\vepsm \in \{-1,1\}^S$ and $c_e$, $c_s$ such that
  \begin{gather*}
    \vphi_*(\xi|_G) = \sum_\vepsm c_\vepsm \eta_{\vepsm,H} + c_e \delta_e +c_s \delta_s
  \end{gather*}
  where $\eta_{\vepsm,H}$ is the Hecke eigenvector in $\lp(H)$ associated with the sign $\vepsm$.  Writing $\eta_\vepsm$ for the Hecke eigenvector in $\lp(W)$ associated with the same sign, we find that
  \begin{gather*}
    \xi_w
    =
    \vphi_*(\xi|_G)_{\vphi(w)}
    =
    (\sum_\vepsm c_\vepsm \eta_{\vepsm,H} +c_s \delta_s + c_e \delta_e )_{\vphi(w)}
    =
    (\sum_\vepsm c_\vepsm \eta_\vepsm +c_s \delta_s + c_e \delta_e)_w
  \end{gather*}
for all $w \in G$ (note that the second equality follows from an explicit check, remembering that $s\in S$ is viewed above first as an element of $H$, and then as an element of $G\subset W$).  Define  a formal power series $\zeta = \xi - \sum_\vepsm c_\vepsm \eta_\vepsm - c_e \delta_e$. We will show in the remainder of the proof that $ \zeta_w=0$ for all $w \in W$. Begin by noting that as $\zeta_s = c_s$ and $\zeta|_{G \setminus \{s\}} = 0$, we have in particular that $\zeta_{st}=0$, so by Lemmas \ref{lem:restriction-double-coset} and \ref{lem:adding-letters} also $\zeta_{tst}=0$ holds.  Then however Lemmas \ref{lem:commutation-relation} and \ref{lem:adding-letters} imply that $c_s = \zeta_s= \zeta_{tst}=0$.

 Let $w \in W$ be some word such that $\ell_s(w) \geq 1$ and $\pos_s(w) \geq 2$.  Assume that $\zeta_x=0$ for all $x \in W$ such that $\ell_s(x) = \ell_s(w)$ and $\pos_s(x) < \pos_s(w)$.  Let $v$ be a letter of $w$ such that $\pos_s(vw) < \pos_s(w)$.  In particular, $w$ starts with $v$.  If $w$ also ends with $v$, then Lemmas \ref{lem:commutation-relation} and \ref{lem:adding-letters} and the inductive assumption say that
  \begin{gather*}
    \zeta_w = \zeta_{vwv} - p_v \zeta_{vw} =  0 %
    \eqstop
  \end{gather*}
  If $w$ does not end with $v$, then the same lemmas imply that
  \begin{gather*}
    \zeta_w = \zeta_{vwv} =0%
    \eqstop
  \end{gather*}

  Let $w \in W$ be some word with $\ell_s(w) \geq 2$ that starts and ends with $s$.  Assume that $\zeta_x = 0$ for all $x \in W$ such that $\ell_s(x) < \ell_s(w)$.  Invoking Lemmas \ref{lem:commutation-relation} and \ref{lem:adding-letters}, we find that
  \begin{gather*}
    \zeta_w = \zeta_{sws} - p_s \zeta_{ws} = 0 %
    \eqstop
  \end{gather*}

  Let $w \in W$ be some word such that $\ell_s(w) \geq 2$ and starting with $s$ but not ending with $s$. We will now proceed with the induction with respect to the place of last occurrence of $s$, which is $\pos_s(w^{-1})$.  Assume that $\zeta_x=0$ for all $x \in W$ such that $\ell_s(x) = \ell_s(w)$ which start with $s$ and satisfy the condition $\pos_s(x^{-1}) < \pos_s(w^{-1})$.  %
  First consider the case that $w$ ends with some letter $v \neq t$ such that $\pos_s((wv)^{-1}) < \pos_s(w^{-1})$.  If $w$ also starts with $v$, then $\ell_v(w) \geq 2$ and Lemmas \ref{lem:commutation-relation} and \ref{lem:adding-letters} say that
  \begin{gather*}
    \zeta_w = \zeta_{vwv} - p_v \zeta_{wv} = 0 %
    \eqcomma
  \end{gather*}
  since $s$ and $v$ commute.  If $w$ does not start with $v$, then the same lemmas say that
  \begin{gather*}
    \zeta_w = \zeta_{vwv} = 0 %
    \eqstop
  \end{gather*}
  So we remain with the case that $w$ ends with $t$.  We distinguish two cases.  If $q_t \leq q_u$, we observe that $\veps_t = 1$ by Lemma \ref{lem:hecke-ev-exists} applied to $H$.  So we can make use of the fact that the double coset $\lang t,u \rang w \lang t,u \rang$ is non-degenerate and apply Lemmas \ref{lem:restriction-double-coset} and \ref{lem:adding-letters} with the inductive assumption to see that
  \begin{gather*}
    \zeta_w = q_t^{-\frac{1}{2}} \zeta_{wt} = 0 %
    \eqstop
  \end{gather*}
On the other hand if $q_t > q_u$, Lemma \ref{lem:hecke-ev-exists} implies that $\veps_u = 1$.  We can thus invoke Lemmas \ref{lem:restriction-double-coset} and \ref{lem:adding-letters} to assume without loss of generality that $w$ starts with $u$. Further consider the word $stswt$, which arises from $w$ by rotating $t$ from the end to the start passing it past $s$.  Lemmas \ref{lem:restriction-double-coset} and \ref{lem:adding-letters} implies that
  \begin{gather*}
    \zeta_w = \zeta_{stswt} = 0 %
    \eqstop
  \end{gather*}
  
  Let us now take $w \in W$ such that $\ell_s(w) \geq 2$ and $\pos_s(w) \geq 2$.  We assume that $\zeta_x=0$ for all $x \in W$ such that $\ell_s(x) = \ell_s(w)$ and $\pos_s(x) < \pos_s(w)$.  Let $v$ be some letter such that $\pos_s(vw) < \pos_s(w)$.  In particular $v$ is a starting letter of $w$.  If $w$ also ends with $v$, then Lemmas \ref{lem:commutation-relation} and \ref{lem:adding-letters} say that
  \begin{gather*}
    \zeta_w = \zeta_{vwv} - p_v \zeta_{wv} = 0 %
    \eqstop
  \end{gather*}
  If $w$ does not end with $v$, then the same lemmas say that
  \begin{gather*}
    \zeta_w = \zeta_{vwv} =0 %
    \eqstop
  \end{gather*}
\end{proof}

\subsection{Proof of Theorems {\ref{thm:factoriality}} and Corollary {\ref{cor:spherical-representations}}}
\label{sec:proofs}

\begin{theorem}
  \label{thm:hecke-central-vectors-hecke-ev}
  Let $(W,S)$ be an irreducible right-angled Coxeter group with at least three generators.  Let $\qm \in (0,1]^S$ and $r \in [1,\infty)$.  Then every Hecke central vector for the multiparameter $\qm$ in $\lp(W)$ is a linear combination of Hecke eigenvectors and $\delta_e$.
\end{theorem}
\begin{proof}
  We prove this theorem by induction on the rank of $W$.  The case of three generators is covered by Lemmas \ref{lem:rank-three-straight-order}, \ref{lem:rank-three-free-generator-largest} and \ref{lem:rank-three-free-generator-smallest}.  So assume that $|S| = n \geq 4$ and that every Hecke central vector for an irreducible Coxeter group of rank at most $n - 1$ is a linear combination of Hecke eigenvectors and $\delta_e$.

  Denote by $\Gamma$ the Coxeter graph of $(W,S)$ and by $d = \mathrm{diam}(\Gamma)$ the diameter of $\Gamma$. For $s, s' \in S$ we write $s \sim s'$ if $\Gamma$ contains an edge connecting $s$ with $s'$.

  \noindent \textbf{Case 1}: $d = 1$. In this case, $\Gamma$ is a complete graph.  So we can find three pairwise different vertices $s,t,u$ of $\Gamma$ such that $q_s \leq q_t$.  Thanks to the induction hypothesis, Lemma \ref{lem:remove-arrows-head} applies.

  \noindent \textbf{Case 2}: $d \geq 2$. In this case we will distinguish several possibilities.
  
   \noindent \textbf{Subcase 2.1}: Suppose that there exists a vertex $s\in S$ which is a leaf, i.e.\ it has degree $1$. Fix its unique adjacent vertex $t$. %
  If $q_s \leq q_t$, we can apply Lemma \ref{lem:remove-arrows-head} to conclude the induction.  Otherwise, Lemma \ref{lem:remove-edge} can be applied. %

  \noindent \textbf{Subcase 2.2}: Suppose that $\Gamma$ admits no leaves and there exists a triple of vertices $o,s, s'\in S$ such that $s \sim s'$, the distance from $o$ to $s$ equals $d$, and so does the distance from $o$ to $s'$. Then $\Gamma$ remains connected upon removing either of the vertices $s$ or $s'$, as follows from the fact that $d$ is the diameter of $\Gamma$.  Hence the induction hypothesis ensures that Lemma \ref{lem:remove-arrows-head} can be applied in a way depending on whether $q_s \leq q_{s'}$ or $q_{s'} \leq q_s$.

  \noindent \textbf{Subcase 2.3}: Suppose that $\Gamma$ admits no leaves and we have a vertex $o \in S$ such that for every vertex $s \in S$ lying at distance $d$ from $o$ all the neighbours of $s$ lie at distance $d-1$ from $o$. Fix then such $s\in S$ and one of its neighbours, $s'\in S$.  Using the fact that every vertex at distance $d$ from $o$ has at least two neighbours at distance $d-1$ from $o$, we see that $\Gamma$ remains connected upon removing either of the vertices $s$ or $s'$.  Hence again the induction hypothesis ensures that Lemma \ref{lem:remove-arrows-head} can be applied in a way depending on whether $q_s \leq q_{s'}$ or $q_{s'} \leq q_s$.

\end{proof}

\begin{proof}[Proof of Theorem \ref{thm:factoriality}] Denote the canonical duality between $\ell^{r'}(W)$ and $\ell^r(W)$ by $\langle \cdot, \cdot \rangle$.
Lemma \ref{commutants} and  Proposition \ref{prop:hecke-eigenvectors} imply that for $\vepsm \in \tilde{C}$  the natural projections onto Hecke eigenvectors in $\lp(W)$, i.e.\ operators of the form 
  $P_\vepsm = \langle \eta_{\vepsm}, \cdot \rangle \eta_{\vepsm}$
   lie in $\cZ(\cN_\qm^r(W))$ (note that the definition of $\tilde{C}$ forces $\eta_{\vepsm} \in \ell^r(W) \cap \ell^{r'}(W)$).  Conversely, if $x \in \cZ(\cN_\qm^r(W))$, then Lemma~\ref{commutants} shows  that $\hat x \in \lp(W)$ is a Hecke central vector. By Theorem \ref{thm:hecke-central-vectors-hecke-ev} it must be a linear combination of Hecke eigenvectors $\eta_\vepsm$  and $\delta_e$.  By Lemma \ref{lem:hecke-ev-exists}, each Hecke eigenvector appearing in this linear combination is an element of $\lp(W)$, so that $\vepsm \in C$. Subtracting from $x$ an appropriate multiple of the identity operator we may assume that $\hat{x} = \sum_{\veps \in C} \alpha_\vepsm \eta_{\vepsm}$. Using now the fact that $x$ commutes with Hecke translations and each $\eta_\vepsm$ is a Hecke eigenvector we see that $x$ must be in fact finite rank, and more specifically of the form
  \[ x (\theta) = \sum_{\vepsm \in C} \langle \xi_{\vepsm}, \theta \rangle \eta_\vepsm, \;\;\; \theta \in \ell^r(W), \]
  where $\xi_{\vepsm}\in \ell^{r'}(W)$ for each $\vepsm \in C$. 
Consider now the adjoint operator $x^*:\ell^{r'}(W)\to \ell^{r'}(W)$. It has the form $x^*(\cdot) =  \sum_{\vepsm \in C} \langle \cdot, \eta_{\vepsm} \rangle \xi_\vepsm$. Recall that the duality is compatible with the action of  the respective Hecke operators -- these associated to elements of $S$ are suitably `selfdual', i.e.\ $(\pi^r_{\qm} (T_s^{(\qm)}))^* = \pi^{r'}_{\qm} (T_s^{(\qm)})$. This, together with the linear independence of $\eta_\vepsm$,
 shows via a direct computation that as a consequence of the commutation relation $x^* \pi^{r'}_{\qm} (T_s^{(\qm)}) = \pi^{r'}_{\qm} (T_s^{(\qm)})x^*$ in fact each $\xi_{\vepsm}$ is either $0$ or a Hecke eigenvector with the same weight as $\eta_{\vepsm}$. This ends the proof -- and explains the appearance of the new set $\tilde{C}$ in the formulation of the theorem.
\end{proof}

We will next use the relation between representations of groups acting on right-angled buildings and generic Hecke von Neumann algebras described in Section \ref{sec:buildings}.  Let us first observe that the identification of Hecke algebras made in Theorem \ref{thm:iwahori-hecke-algebra-is-generic} extends to an identification of Hecke von Neumann algebras, before we prove Corollary \ref{cor:spherical-representations}.
\begin{theorem}
  \label{thm:iwahori-hecke-vn-algebra-is-generic}
    Let $X$ be a locally finite, thick building of type $(W,S)$ and $G \leq \Aut(X)$ be a closed, strongly transitive and type preserving subgroup.  Let $B \leq G$ be an Iwahori subgroup of $G$ and denote by $(d_s)_{s \in S}$ the thickness of $X$.  Let $q_s = d_s^{-1}$, for $s \in S$.  Then $\cN_\qm(W) \cong \rL[G, B]$.
\end{theorem}
\begin{proof}
  Denote by $\vphi$ the isomorphism $\CC_\qm[W] \to \CC[G,B]$ from Theorem \ref{thm:iwahori-hecke-algebra-is-generic}.  The Hecke von Neumann algebra $\cN_\qm(W)$ is the completion of $\CC_\qm[W]$ in the GNS-representation associated with the tracial state satisfying $T_w^{(\qm)} \mapsto \delta_{w,e}$.  This state is recovered as the composition of $\vphi$ with the tracial vector state on $\CC[G,B]$ associated with $\delta_{Be} \in \ltwo(B \backslash G)$.  So we find that $\vphi$ extends to a *-isomorphism $\rL(G,B) \cong \cN_\qm(W)$.
\end{proof}

\begin{proof}[Proof of Corollary \ref{cor:spherical-representations}]
  Since $\lambda_{G,B}(G)' = \rR(G,B) \cong \rL(G,B)$, a combination of Theorem \ref{thm:iwahori-hecke-vn-algebra-is-generic} and Theorem \ref{thm:factoriality} applies to describe the structure of the quasi-regular representation.  Write
  \begin{gather*}
    C = \{ \vepsm \in \{-1,1\}^S \mid \sum_{w \in W} d^{-1}_{w, \vepsm} \text{ is absolutely summable}\}
    \eqstop
  \end{gather*}
  Let us directly observe that $\sum_{w \in W} d_w^{-1} = \infty$ implies $C = \emptyset$.  We find that $\lambda_{G,B} = \pi \oplus \bigoplus_{\vepsm \in C} \mathrm{St}_\vepsm$ for a semi-finite factor representation $\pi$ and pairwise unitarily inequivalent, irreducible representations $\mathrm{St}_\vepsm$, where $\mathrm{St}_\vepsm$ is the restriction of $\lambda_{G,B}$ to the invariant subspace $p_\vepsm L^2(B\backslash G)$, with $p_\veps$ being the projection appearing in Theorem \ref{thm:factoriality}.  We have to show that neither $\pi$ nor $\mathrm{St}_\vepsm$, $\vepsm \in C$ can be finite representations.  We observe that $\mathrm{St}_\vepsm$ are subrepresentations of $\lambda_{G,B}$ and hence infinite dimensional.  Indeed, $B$ is compact since $X$ is locally finite and $G$ is non-compact as it acts (strongly) transitively, so that $\lambda_{G, B}$ is a mixing representation.  Let us now show that $\pi$ is of type ${\rm II}_\infty$.  For a contradiction, let us assume that it is of type ${\rm II}_1$, hence finite.  By \cite[Theorem 15.D.5]{bekkadelaharpe19}, applied to the quotient $G/\ker \pi$ (whose points are separated by our finite type representation), $G/\ker \pi$ has a compact open normal subgroup, say $L$.  Its preimage $N \unlhd G$ must satisfy $NB = G$ by \cite[Proposition 2.5]{tits64} (see alternatively \cite[Lemma 6.61]{abramenkobrown08}).  So $N$ has finite index in $G$, since $G/N \cong B/(B \cap N)$ and $B \cap N$ is open inside the compact group $B$.  So also $L = N / \ker \pi$ has finite index inside $G/\ker \pi$ which implies that the latter group is compact.  We reach a contradiction, since compact groups have no type ${\rm II}$ representations.

Let us assume that $(d_s)_{s \in S}$ is constant and we will show that $\lambda_{G,K}(G)''$ is a factor.  We observe that $C = \emptyset$ or $C = \{\mathbf{1}\}$.  The former case occurs if and only if $\sum_{w \in W} q_{w, \mathbf{1}} = \infty$, and then already $\lambda_{G,B}$ is a factor representation.  Since $\Ltwo(G/K)$ is a $G$-invariant subspace of $\Ltwo(G/B)$ is follows that $\lambda_{G,K}(G)''$ is a quotient of $\lambda_{G,B}(G)''$. Since the latter is a factor, we find that $\lambda_{G,K}(G)'' \cong \lambda_{G,B}(G)''$ is a factor too.  So let us assume that $\sum_{w \in W} q_{w, \mathbf{1}} < \infty$ and write $\St = \St_{\mathbf{1}}$.  We will show that the irreducible representation $\mathrm{St}$ does not admit any non-zero $K$-invariant vector.  We may assume that $B \leq K$ and pick $s \in S$ such that $BsB \subset K$.  Indeed, if $B$ is the stabiliser of $C$, we can pick $K$ to be the stabiliser of a vertex adjacent to the facet labelled by $s$.  Denote by $p$ the minimal projection in $\rL(G,B)$ projecting onto the representation space of $\St$.  It suffices to show that $p_K p = 0$.  Combining the isomorphism of Hecke algebras Theorem \ref{thm:iwahori-hecke-algebra-is-generic} with the calculation of eigenvalues for $T_s^{(\qm)}$ in Proposition~\ref{prop:hecke-eigenvectors}, we find that
  \begin{gather*}
    \mathbb{1}_{BsB} p = -p
    \eqstop
  \end{gather*}
  On the other hand $\mathbb{1}_{BsB} \, p_K = p_K$ follows from the definition of the convolution product on $\contc(G)$, since $BsB \subset K$.  This shows that $p_K p = 0$ must hold.  This shows that $\rL(G,K) = p_K \rL(G, B) p_K$ is a factor and thus also $\lambda_{G,K}(G)''$ is a factor.  
\end{proof}

\begin{example}
  \label{ex:spherical-representation}
  The statement of Corollary \ref{cor:spherical-representations} leaves open the possibility that for a given $G$ an irreducible Iwahori-spherical square integrable representation exists.  For large enough thickness of the underlying building, such a representation will always exist according to Corollary \ref{cor:spherical-representations}.  Indeed, if $(W, S)$ is a Coxeter group and $X$ is the building of type $(W, S)$ and thickness $(d_s)_{s \in S}$, then a closed, strongly transitive and type preserving subgroup of $\Aut(X)$ admits an irreducible Iwahori-spherical square integrable representation if and only if $\sum_{w \in W} q_w < \infty$, where $q_s = d_s^{-1}$.  This applies in particular to the full group of type-preserving automorphisms $\Aut(X)^+$, which has finite index in $\Aut(X)$.

  \begin{itemize}
  \item We first discuss one example in which there always is an irreducible direct summand of $\lambda_{G, B}$.  Let $W = \ZZ/2\ZZ^{ \oplus 2} * \ZZ/2\ZZ$ and $S = \{s,t,u\}$ the set of its standard generators where $u$ is chosen to generate the free copy of $\ZZ/2\ZZ$.  Enumerating elements of $W$ according their $u$-length, one finds that
    \begin{gather*}
      \sum_{w \in W} q_w
      =
      (1 + q_s + q_t + q_{st}) + (1 + q_s + q_t + q_{st})^2 \sum_{n \in \NN_{\geq 1}} (q_s + q_t + q_{st})^{n-1} q_u^n 
      \eqstop
    \end{gather*}
    It follows that $\sum_{w \in W} q_w < \infty$ if and only if $q_u \left( (1 + q_s)(1 + q_t) - 1 \right ) < 1$.  The last inequality is equivalent to $\frac{(1 + d_s)(1 + d_t)}{d_s d_t} < 1 + d_u$, which holds for all allowable thicknesses, since $d_u, d_s, d_t \geq 2$.

    \item Consider then the example $W = \ZZ/2\ZZ^{\oplus 2} * \ZZ/2\ZZ^{\oplus 2}$ with generators $S = \{s,t,u,v\}$ chosen such that the pairs $s,t$ and $u,v$ commute.  Reasoning as before, we calculate
    \begin{gather*}
      \sum_{w \in W} q_w
      =
      (1 + q_s + q_t + q_{st}) + (1 + q_s + q_t + q_{st})^2 \sum_{n \in \NN_{\geq 1}} (q_s + q_t + q_{st})^{n-1} (q_u + q_v + q_{uv})^n
      \eqstop
    \end{gather*}
    Hence $\sum_{w \in W} q_w < \infty$ if and only if
    \begin{gather*}
      \left (\frac{(1 + d_s)(1 + d_t)}{d_sd_t} - 1 \right) \left (\frac{(1 + d_u)(1 + d_v)}{d_ud_v} - 1 \right) < 1
      \eqstop
    \end{gather*}
    This inequality is satisfied for example for all thick buildings of type $(W,S)$ such that $d_s, d_t, d_u, d_v \geq 3$.

    \item We finally also find examples where no irreducible Iwahori-spherical discrete series representation exists.  Taking $W = \ZZ/2\ZZ^{*k}$ for $k \geq 4$ with its set of Coxeter generators $S$ there are $k(k-1)^{n-1}$ words of length $n \geq 1$ in $W$.  Hence, for a building of type $(W, S)$ and uniform thickness $d$, we find
    \begin{gather*}
      \sum_{w \in W} q_w
      =
      1 + \sum_{n \in \NN_{\geq 1}} k(k-1)^{n-1}\frac{1}{d^n}
      \eqcomma
    \end{gather*}
    which is summable if and only if $d > k - 1$.
  \end{itemize}
\end{example}

\section{Graph products of unitary representations}
\label{sec:graph-products-representations}

In this section we will develop a systematic point of view on the isomorphism $\CC[W] \cong \CC_\qm[W]$ considered in \cite[p.358, Note 19.2]{davis08} and \cite[Proposition 4.2]{caspersklisselarsen21}.  We will consider a special class of what should be thought of as graph products of representations.  The construction of such representations follows a standard pattern known from free products of Hilbert spaces \cite{avitzou82, voiculescu85}, which in turn was the basis of the introduction of graph product of Hilbert spaces as considered in \cite[Section~2.1]{caspersfima17}.

We will briefly describe the graph product of representations.  Using the notation and terminology of \cite{caspersfima17}, let $\Gamma$ be a simplicial graph, let $G_v$, $v \in V\Gamma$ be a family of groups index by the vertices of $\Gamma$ and denote by $\prod_{v, \Gamma} G_v$ the graph product, obtained as the quotient of the free product $G_\Gamma = *_{v \in \Gamma} G_v$ by the commutation relations $[g,h] = e$ for $g \in G_v$, $h \in G_w$ and $v \sim_\Gamma w$.  Suppose that for every $v \in V\Gamma$ we are given a unitary representation $\pi_v:G_v \to B(H_v)$ and a unit vector $\xi_v \in H_v$. Let $H$ denote the graph product of Hilbert spaces associated to the family $(H_v, \xi_v)_{v \in \Gamma}$. It is then easy to see that using the universal property of $G_\Gamma$ and injective unital *-homomorphisms $\lambda_v:B(H_v) \to B(H)$ described in \cite[Section 2]{caspersfima17} we obtain a representation $\pi_\Gamma: G_\Gamma \to B(H)$ such that $\pi_\Gamma|_{G_v}$ is unitarily equivalent to a multiple of the representation $\pi_v$. The resulting representation does depend on the choice of the vectors $\xi_v$, as noted already in \cite{avitzou82}.

Given a right-angled Coxeter system $(W, S)$ and $s \in S$, we have a natural bijection $W \cong W/\langle s \rangle \times \langle s \rangle$, sending $w\in W$ to $([w],s)$ if $|ws|< |w|$ and to $([w],e)$ if $|ws|> |w|$.   Applying this bijection to the natural basis of $\ltwo(W)$, we obtain a unitary $U_s: \ltwo(W) \to \ltwo(W/\langle s \rangle) \ot \CC^2$. Note that we view $\ZZ/2\ZZ$ as the additive group $\{0,1\}\, \textup{mod} 2$.
\begin{proposition}
  \label{prop:graph-product-representations}
  Let $(W, S)$ be a right-angled Coxeter system and let $\pi_s: \ZZ/2\ZZ \to \cU(2)$, $s \in S$ be a family of unitary representations.  The formula
  \begin{gather*}
    \pi(s) = (U_s)^{-1} (\id_{\ltwo(W/\langle s \rangle)} \otimes \pi_s(1)) U_s
    \qquad
    \text{ for } s \in S 
  \end{gather*}
  defines a unitary representation $\pi: W \to \cU(\ltwo(W))$.
\end{proposition}
\begin{proof}
  From the definition it is clear that $\pi$ defines a unitary representation $\tilde \pi$ of the free product group $\tilde W = *_S \ZZ/2\ZZ$.  We have to show that it descends to a representation of $W$ by verifying the commutation relation $[\tilde \pi(s), \tilde \pi(t)] = 0$ for all $s,t \in S$ that commute.  So fix different commuting elements $s,t \in S$.  Since $\langle s,t \rangle = \langle s \rangle \times \langle t \rangle \subset W$, the $s$ and $t$-length on $W$ give rise to a bijection $W \cong W/ \langle s,t \rangle \times \ZZ/2\ZZ \times \ZZ/2\ZZ$, which defines a unitary $U: \ltwo(W) \to \ltwo(W/ \langle s,t \rangle) \otimes \CC^2 \otimes \CC^2$.  We have
  \begin{align*}
    \tilde \pi(s) & = U^{-1} (\id \otimes \pi_s(1) \otimes \id) U \\
    \tilde \pi(t) & = U^{-1} (\id \otimes \id \otimes \pi_t(1)) U
               \eqstop
  \end{align*}
  So indeed $[\tilde \pi(s), \tilde \pi(t)] = 0$, showing that $\pi$ is well-defined.
\end{proof}

If we wish to put the construction above explicitly in the framework of graph products of representations mentioned before the lemma, we should only note the natural identifications, exploited for example in \cite{caspers20}, of $W$ with the graph product of order 2 groups, and of $\ell^2(W)$ with the graph product Hilbert space $H_\Gamma$ (where each $H_s$ for $s \in V\Gamma=S$ is the two-dimensional vector space $\ltwo(\langle s \rangle)$ with the chosen unit vector $\delta_e$) and verify that the above construction is identical with the graph product representation.

\begin{example}
  \label{ex:racg-representations}
  Let us understand graph products of two-dimensional unitary representations arising in the study of right-angled Coxeter groups in more detail.  The spectral theorem provides the decomposition of representation spaces
  \begin{gather*}
    \Rep(\ZZ/2\ZZ, \ltwo(\ZZ/2\ZZ)) = \Rep(\ZZ/2\ZZ, \CC^2) = \{1\} \sqcup \{-1\} \sqcup \bS^2
    \eqstop
  \end{gather*}
  The two isolated components arise from the trivial representation and the sign representation.  The component $\bS^2 \cong \cU(2)/\bT^2$ arises from a direct sum decomposition of $\CC^2$ into a pair of orthogonal subspaces on which $\ZZ/2\ZZ$ acts trivially and by the sign representation, respectively.  An explicit description of this component is obtained after observing that the nontrivial element $1 \in \ZZ/2\ZZ$ must act by a self-adjoint trace-free matrix
  \begin{gather*}
    \begin{pmatrix}
      a & 0 \\
      0 & -a
    \end{pmatrix}
    +
    \begin{pmatrix}
      0 & \ol{z} \\
      z & 0
    \end{pmatrix}
    \eqcomma
  \end{gather*}
  for which the values $a \in \RR$ and $z \in \CC$ must satisfy $a^2 + |z|^2 = 1$.

  Applying Proposition \ref{prop:graph-product-representations} to a right-angled Coxeter system $(W, S)$, we obtain for each multiparameter $(\am, \zm) = (a_s, z_s)_{s \in S} \in (\bS^2)^{S}$ a unitary representation $\tilde \lambda_{\am,\zm}$ of $W$ on $\ltwo(W)$.  Denoting by
  \begin{gather*}
    \sigma(s) \delta_w
    =
    (-1)^{|w|_s}\delta_w
    =
    \begin{cases}
      \phantom{-} \delta_w & \text{if } |sw| > |w| \\
      - \delta_w & \text{if } |sw| < |w| \eqstop
    \end{cases}
  \end{gather*}
  the alternating sign representation of $W$ on $\ltwo(W)$, we have $\tilde \lambda_{\am,\zm}(s) = a_s \sigma(s) + z_s \lambda(s)$, if $\zm \in \RR^S$.  In particular, if $(\am, \zm) = (0, 1)_{s \in S}$, then $\tilde \lambda_{\am,\zm} = \lambda$ is the left-regular representation of $W$.

  We observe that given $(\am, \zm) \in (\bS^2 \setminus \{(1,0), (-1,0)\})^S$ we can put
  \begin{align*}
    \vphi_s & = \frac{z_s}{|z_s|} \\
    \vphi_w & = \vphi_{s_1} \dotsm  \vphi_{s_n} \qquad \text{ for a reduced expression } w = s_1 \dotsm s_n
  \end{align*}
  and obtain a unitary equivalence $\tilde \lambda_{\am,\zm} \cong \tilde \lambda_{\am, |\zm|}$ implemented by the unitary $U: \ltwo(W) \to \ltwo(W)$ which satisfies $U \delta_w = \vphi_w \delta_w$.  We will thus write in what follows $\tilde \lambda_\am = \tilde \lambda_{\am, \zm}$ for $\am \in (-1,1)^S$ putting $z_s = \sqrt{1 - a_s^2}$ for all $s \in S$.
\end{example}

The identification of Hecke operator algebras of right-angled Coxeter groups with operator algebras generated by their unitary representations was already observed in \cite[p.358, Note 19.2]{davis08} and in \cite[Proposition 4.2]{caspersklisselarsen21}.  These proofs admit a convenient reformulation by means of the graph product, which we provide for illustration. We formulate it for the Hilbert space case, but it could be also stated in the $\lp$-framework when taking into account bounded, non-isometric representations on $\lp(W)$.
\begin{proposition}
  \label{prop:hecke-algebras-from-representations}
  Let $(W,S)$ be a right-angled Coxeter system and $\am \in (-1,1)^S$.  Define $\qm \in (0,1]^S$ by
  \begin{gather*}
    q_s^{1/2} = \frac{1 - |a_s|}{z_s}
    \eqstop
  \end{gather*}
  Then $\tilde \lambda_\am(W)'' \cong \cN_\qm(W)$.
\end{proposition}
\begin{proof}
  For $\epsilonm \in \{-1,1\}^S$ given by $\epsilon_s = \sign(a_s)$, we put $\epsilonm \zm = (\epsilon_s z_s)_{s \in S}$.  Then $\epsilon_s \tilde \lambda_\am(s) = \tilde \lambda_{|\am|, \epsilon \zm}(s)$ for all $s \in S$ and $\tilde \lambda_{|\am|, \epsilonm \zm} \cong \tilde \lambda_{|\am|}$.  This allows us to reduce the considerations to the case where $a_s \geq 0$ for all $s \in S$.  Since both the Iwahori-Hecke algebras and the representations $\tilde \lambda_a$ arise as graph products, it suffices to consider the case $W = \ZZ/2\ZZ$ and match entries of $2 \times 2$-matrices. The ansatz $\tilde \lambda_\am(s) = c1 + d T_s^{(\qm)}$ for real parameters $c,d \in \RR$ yields the equality
  \begin{gather*}
    \begin{pmatrix}
      a_s & z_s \\
      z_s & -a_s
    \end{pmatrix}
    =
    \begin{pmatrix}
      c & d \\
      d & c + d \frac{q_s - 1}{q_s^{1/2}}
    \end{pmatrix}
    \eqstop
  \end{gather*}
  Employing the identity $z_s = \sqrt{1 - a_s^2}$, a short calculation yields $q_s^{1/2} = \frac{1 - a_s}{z_s}$, finishing the proof of the proposition.
\end{proof}

We would like to note here that the above result points the way to constructing and studying interesting deformations of group algebras associated say to RAAGs via the graph products of representations, using the graph product structure of the groups in question and replacing the left regular representation of the individual components by their `perturbations', as was done above with the representations of order 2 groups.

By Theorem \ref{thm:factoriality}, we know that $\tilde \lambda_a$ is a direct sum of a finite factor representation and a finite (possibly empty) collection of one-dimensional representations.  Let us first identify the one-dimensional summands, slightly refining the picture of Hecke eigenvectors as considered in Proposition \ref{prop:hecke-eigenvectors} and allowing us to pick up the signs of a parameter $\am \in (-1,1)^S$.  %
\begin{proposition}
  \label{prop:eigenvectors}
  Let $(W, S)$ be a right-angled Coxeter group, $\am \in (-1,1)^S$ and set for each $s \in S$ 
  \begin{gather*}
    q_s^{\frac{1}{2}}
    =
    \sign(a_s) \frac{1 - |a_s|}{z_s}
    =
    \begin{cases}
      \frac{1 - a_s}{z_s} & \text{ if } a_s \geq 0 \\
      - \frac{1 + a_s}{z_s} & \text{ if } a_s < 0
      \eqstop
    \end{cases}
  \end{gather*}
  Let $\vepsm \in \{-1,1\}^S$, assume that $\sum_{w \in W} q_{w, \vepsm}$ is absolutely summable and put $\eta_\vepsm = \sum_{w \in W} q_{w, \vepsm}^{\frac{1}{2}} \delta_w \in \ltwo(W)$.  Then
  \begin{gather*}
    \tilde \lambda_\am(s)\eta_\vepsm = \veps_s \sign(a_s) \eta_\vepsm
    \eqstop
  \end{gather*}
\end{proposition}
\begin{proof}
  Since $\tilde \lambda_\am(s) = \epsilon_s \tilde \lambda_{|a|, \epsilon z}(s)$ and the unitary conjugating $\tilde \lambda_{|a|, \epsilon z}$ to $\lambda_{|a|}$ respects the definition of $q_s$, we may assume that $a_s \geq 0$ for all $s \in S$.  Then 
  \begin{gather*}
    \tilde \lambda_\am(s) = a_s1 + z_s T_s^{(\qm)}
  \end{gather*}
  for all $s \in S$ follows as in the proof of Proposition \ref{prop:hecke-algebras-from-representations}.  Combining this with the calculation of Hecke eigenvalues from Proposition \ref{prop:hecke-eigenvectors}, we find that
  \begin{gather*}
    \tilde \lambda_\am(s) \eta_\vepsm = (a_s + z_s \veps_s q_s^{\veps_s \frac{1}{2}}) \eta_\vepsm
    \eqstop
  \end{gather*}
  We further calculate
  \begin{gather*}
    a_s + z_s q_s^{\frac{1}{2}} = a_s + 1 - a_s = 1
  \end{gather*}
  and
  \begin{gather*}
    a_s - z_s q_s^{-\frac{1}{2}}
    =
    a_s - \frac{z_s^2}{1 - a_s}
    =
    \frac{(1 - a_s)a_s - (1 - a_s^2)}{1 - a_s}
    =
    -1
    \eqstop
  \end{gather*}
\end{proof}

\begin{notation}
  \label{not:reduced-representation}
  Let $(W, S)$ be an irreducible, right-angled Coxeter system with at least three generators and let $\am \in (-1, 1)^S$.  We denote by $\lambda_{\am}$ the weakly mixing part of $\tilde \lambda_{\am}$, that is the restriction of $\tilde \lambda_{\am}$ to the orthogonal complement of its finite dimensional irreducible  subrepresentations (which are at most one-dimensional).  By Theorem \ref{thm:factoriality}, Proposition \ref{prop:hecke-algebras-from-representations} and Proposition \ref{prop:eigenvectors}, all $\lambda_\am$ are finite factor representations of $W$.
\end{notation}

We will finally record some properties of characters of our representations.
\begin{lemma}
	\label{lem:calcuation-character}
	Let $(W, S)$ be a right-angled Coxeter system, let $\am \in (-1,1)^S$ and let $\tau_\am$ be the character of $(\tilde \lambda_\am, \delta_e)$.
	\begin{itemize}
		\item For $s \in S$, we have $\tau_\am(s) = a_s$.
		\item For $s,t \in S$ different letters, we have $\tau_\am(st) =a_sa_t$.
		\item For $s,t,u \in S$ pairwise different letters, we have $\tau_\am(stu) =a_sa_ta_u$.
		\item If $s,t \in S$ do not commute, then $\tau_\am(stst) = a_s^2 + a_t^2 - a_s^2a_t^2$.
	\end{itemize}
\end{lemma}
\begin{proof}
	It is easy to note that if $w = s_1 \dotsm s_n$ is a reduced word in $W$, then the computation of $\tau_\am(w) = \langle \tilde \lambda_\am(w) \delta_e, \delta_e \rangle$ amounts to summing over subwords of $w$ which are trivial in $W$ certain coefficients determined as products of $\pm a_s$ and $z_s$, where $s$ runs over the letters in the subword.  Writing these expressions down explicitly yields the statement of the lemma. In particular, $\tau_\am(s) = a_s$ and $\tau_\am(st) = a_sa_t$ and $\tau_\am(stu) =a_sa_ta_u$ are clear.  The last statement follows from the calculation
	\begin{gather*}
		\tau_\am(stst)
		=
		a_t^2 a_s^2 + a_t^2 z_s^2 + z_t^2 a_s^2
		=
		a_t^2 a_s^2 + a_t^2 (1 - a_s^2) + (1 - a_t^2) a_s^2
		=
		a_s^2 + a_t^2 - a_s^2 a_t^2
		\eqstop
	\end{gather*}
\end{proof}

We can directly proceed to the proofs of Theorem \ref{thm:unitary-equivalence} and Corollary \ref{cor:traces}.
\begin{proof}[Proof of Theorem \ref{thm:unitary-equivalence}]
  The fact that $\cN_\qm(W) = \lambda_\am(W)'' \oplus A_\am$ follows from Proposition \ref{prop:hecke-algebras-from-representations}.  By Theorem~\ref{thm:factoriality} there is a neighbourhood $U$ of $0 \in (-1,1)^S$ such that $\tilde \lambda_\am$ is a (finite) factor representation for all $\am \in U$. If $\lambda_\am \cong \lambda_\bm$ for multiparameters $\am, \bm \in (-1,1)^S$, then their characters must agree.  So Lemma \ref{lem:calcuation-character} implies that $\am = \bm$.
\end{proof}

\begin{proof}[Proof of Corollary \ref{cor:traces}]
  Consider a neighbourhood $U \subset (-1,1)^S$ of zero such that $\tilde \lambda_\am = \lambda_\am$ is a factor representation for all $\am \in U$.  Theorem \ref{thm:unitary-equivalence} says that the representations $\lambda_\am$, $\am \in U$ are pairwise unitarily inequivalent.  Also $\am \mapsto \tilde \lambda_\am = \lambda_\am$ is continuous (with respect to the topology of pointwise convergence on $\textup{Char}(W)$) and $\lambda_0 = \lambda$ is the regular representation.  So the corollary follows from the fact that that $\am \mapsto \lambda_\am$ is a homeomorphism onto the image of a neighbourhood of $\mathbf{0} \in \mathbb{R}^{|S|}$.
\end{proof}

In the remainder of this section, we show that Theorem \ref{thm:unitary-equivalence} admits an extension to a larger set of multiparameters.  We believe that the representations $\lambda_\am$ are pairwise unitarily inequivalent for all $\am \in (-1,1)^S$.

\setlength{\parindent}{0em}

\begin{lemma}
  \label{lem:unitary-equivalence-basic-cases}
  Let $(W,S)$ be an irreducible, right-angled Coxeter group and let $\am,\bm \in (-1, 1)^S$. Denote by $\tau_{\am}, \tau_{\bm}$ the characters of $(\tilde \lambda_{\am}, \delta_e)$ and $(\tilde \lambda_{\bm}, \delta_e)$, respectively. Assume that there are $c, d \in \RR$ such that
  \begin{gather*}
    \tau_\am = c \sigma_\am + d \sigma_\bm + (1 - c - d)\tau_\bm
    \eqstop
  \end{gather*}
  If there is $s \in S$ such that $a_s = b_s$, then $\am = \bm$.
\end{lemma}
\begin{proof}
  The statement is trivial if $S$ contains only one element.  Without loss of generality, we may assume that $a_s \geq 0$, which simplifies notation.  We apply Lemma \ref{lem:calcuation-character} to obtain
  \begin{gather*}
    a_s = c + d + (1 - c - d)b_s = c + d + (1 - c - d)a_s
    \eqstop
  \end{gather*}
  Simplifying this equality, we obtain $a_s (c + d) = c + d$, which implies $c + d = 0$, since $a_s \neq 1$.  Let now $t \in S \setminus s$ be arbitrary.  If $\sign(a_t) = \sign(b_t)$, Lemma \ref{lem:calcuation-character} yields
  \begin{gather*}
    a_t = \sign(a_t)(c + d) + (1 - c - d)b_t = b_t
    \eqstop
  \end{gather*}
  If $\sign(a_t) \neq \sign(b_t)$, then the identity $c + d = 0$ implies
  \begin{gather*}
    a_t = \sign(a_t)(c - d) + (1 - c - d)b_t = \sign(a_t)2c + b_t
    \eqstop
  \end{gather*}
  Similarly,
  \begin{gather*}
    a_sa_t = \sign(a_t)2c + a_sb_t
    \eqstop
  \end{gather*}
  Taking the difference of these identities gives us
  \begin{gather*}
    (1 - a_s) a_t = (1 - a_s)b_t
    \eqcomma
  \end{gather*}
  which allows to conclude $a_t = b_t$, because $a_s \neq 1$.  Since $t \in S \setminus s$ was arbitrary, this concludes the proof of the lemma.
\end{proof}

\begin{proposition}
  \label{prop:unitary-equivalence-single-parameter}
  Let $(W,S)$ be an irreducible, right-angled Coxeter system with at least three generators.  Consider the set $U$ of all $\am \in (-1,1)^S$ such that absolute convergence of $\sum_w q_{w,\vepsm}^{\frac{1}{2}}$ (with $\qm$ associated to $\am$ as in Proposition \ref{prop:hecke-algebras-from-representations}) implies $\vepsm = \textbf{1}$. If  $\am, \bm \in U$ and $\lambda_\am \cong \lambda_\bm$, then $\am = \bm$.
\end{proposition}
\begin{proof}
  Denote by $\tau_\am,\tau_\bm$ the characters of $(\tilde \lambda_\am, \delta_e)$ and $(\tilde \lambda_\bm, \delta_e)$, respectively.  Then Proposition \ref{prop:eigenvectors} implies that there are scalars $c,d \in \RR$ such that $c + d \in [0,1)$ and
  \begin{gather*}
    \tau_\am = c \sigma_\am + d \sigma_\bm + (1 - c - d)\tau_\bm
    \eqstop
  \end{gather*}
  Replacing $a$ and $b$ by $(|a_s|)_{s \in S}$ and $(\sign(a_s)b_s)_{s \in S}$, respectively, we may assume that $a_s \geq 0$ for all $s \in S$.  We distinguish three cases.
  
  \textbf{Case 1}.  There are two letters $s,t \in S$ such that $b_s, b_t \geq 0$. \\
  Putting $\gamma = c + d$, Lemma \ref{lem:calcuation-character} implies that
  \begin{align*}
    a_s & = \gamma + (1 - \gamma)b_s \\
    a_t & = \gamma + (1 - \gamma)b_t \\
    a_sa_t & = \gamma + (1 - \gamma)b_sb_t
    \eqstop
  \end{align*}
  Substituting the first two expressions into the third, and simplifying, we obtain
  \begin{align*}
    0 & =
        (\gamma^2 - \gamma) + \gamma(1 - \gamma)(b_t + b_s) + ((1 - \gamma)^2 - (1 - \gamma))b_s b_t \\
      & =
        \gamma (\gamma -  1) (b_s - 1)(b_t - 1)
        \eqstop
  \end{align*}
  Since $\gamma, b_s, b_t \neq 1$, we find that $\gamma = 0$ and thus $a_s = b_s$.  So Lemma \ref{lem:unitary-equivalence-basic-cases} shows that $\am = \bm$.

  \textbf{Case 2}. There is exactly one $u \in S$ such that $b_u \geq 0$. \\
  We will obtain a contradiction in this case.  Take $s \in S \setminus \{u\}$.  Then Lemma \ref{lem:calcuation-character} provides the identities
  \begin{align*}
    a_s & = c - d + (1 - c - d)b_s \\
    a_ua_s & = c - d + (1 - c - d)b_ub_s
             \eqcomma
  \end{align*}
  which we subtract in order to obtain
  \begin{gather*}
    0 \leq a_s(1 - a_u) = (1 - c - d) b_s ( 1 - b_u)
    \eqstop
  \end{gather*}
  Since $b_s < 0$ and $1 - b_u > 0$, this implies $1 - c - d \leq 0$.  The latter is a contradiction to $c + d \in [0,1)$.

  \textbf{Case 3}. We have $b_s < 0$ for all $s \in S$. \\
  Again we aim for a contradiction.  Take three pairwise different letters $s,t,u \in S$.  From Lemma \ref{lem:calcuation-character}, we obtain that
  \begin{align*}
    a_s & = c - d + (1 - c - d)b_s \\
    a_t & = c - d + (1 - c - d)b_t \\
    a_ua_s & = c + d + (1 - c - d)b_ub_s \\
    a_ua_t & = c + d + (1 - c - d)b_ub_t
             \eqstop
  \end{align*}
  We subtract the second identity from the first one and the fourth from the third one, and we obtain
  \begin{align*}
    a_s - a_t & = (1 - c - d)(b_s - b_t) \\
    a_u(a_s - a_t) & = (1 - c - d)b_u(b_s - b_t)
                     \eqstop
  \end{align*}
  Since $a_u \neq b_u$, this implies $a_s = a_t$.  Because $c + d \neq 1$, we also obtain the identity $b_s = b_t$. Similarly $a_s = a_u$.

  Let us now take pairwise different letters $s,t, u \in S$.  Evaluating $\tau_\am = c\sigma_\am + d \sigma_\bm + (1 - c - d)\tau_\bm$ on the reduced words $s$ and $stu$, we obtain
  \begin{align*}
    a_s & = c - d + (1 - c - d)b_b \\
    a_s^3 & = a_sa_ta_u = c - d + (1 - c - d)b_sb_tb_u = c - d + (1 - c - d)b_s^3
            \eqstop
  \end{align*}
  Subtracting the second from the first identity we obtain
  \begin{gather*}
    a_s - a_s^3 = (1 - c  - d)(b_s - b_s^3)
    \eqstop
  \end{gather*}
  Since $a_s \in [0,1)$ and $b_s \in (-1,0)$, this implies that $1 - c - d < 0$, which is a contradiction to $c + d \in [0,1)$.
\end{proof}

The last proposition applies for example to the single parameter case, in the sense that the set $U$ contains all the constant sequences.

\printbibliography

\vspace{2em}
{\small
  \parbox[t]{200pt}
  {
    Sven Raum \\
    Department of Mathematics \\
    Stockholm University \\
    SE-106 91 Stockholm \\
    Sweden \\[1em]
    and \\[1em]
    Institute of Mathematics of the \\ Polish Academy of Sciences \\
    ul.\ \'Sniadeckich 8 \\
    00-656 Warszawa \\
    Poland \\[1em]
    {\footnotesize raum@math.su.se}
  }
  \hspace{3em}
  \parbox[t]{200pt}
  {
    Adam Skalski \\
    Institute of Mathematics of the \\ Polish Academy of Sciences \\
    ul.\ \'Sniadeckich 8 \\
    00-656 Warszawa \\
    Poland \\
    {\footnotesize a.skalski@impan.pl}
  }
}

\end{document}